\documentclass{article}

\usepackage{a4wide}
\usepackage[dvipsnames]{xcolor}
\usepackage{amsmath,amsthm}

\usepackage[english]{babel}

\definecolor{myred}{HTML}{FF3D3D}
\definecolor{mycyan}{HTML}{0474BE}

\usepackage{amssymb}

\usepackage{mathtools}

\usepackage{booktabs}

\usepackage{units}

\usepackage{subfig}

\numberwithin{equation}{section}

\usepackage{units}

\newcommand{\R}{\mathbb{R}}
\newcommand{\Z}{\mathbb{Z}}
\newcommand{\eps}{\varepsilon}
\newcommand{\norm}[1]{\left\lVert#1\right\rVert}
\newcommand{\abs}[1]{\left\lvert#1\right\rvert}

\newcommand{\sign}{\operatorname{sgn}}

\newcommand{\Dx}{{\Delta x}}
\newcommand{\Dt}{{\Delta t}}

\newcommand{\hf}{\unitfrac{1}{2}}

\newcommand{\cell}{\mathcal{C}}

\newcommand*\diff{\mathop{}\!\mathrm{d}}

\usepackage{hyperref}

\usepackage{tikz}
\usepackage{pgfplots}
\pgfplotsset{width=.36\textwidth,compat=1.12}
\pgfplotsset{every tick label/.append style={font=\footnotesize}}
\pgfplotscreateplotcyclelist{mycycle}{
myblue,
myblue,
scarletred1
}
\pgfplotscreateplotcyclelist{mycycle2}{
{color=myblue, mark=*},
{plum1, mark=triangle*},
{scarletred1, mark=square*}
}

\definecolor{skyblue1}{rgb}{0.447,0.624,0.812}
\definecolor{plum1}{rgb}{0.678,0.498,0.659}
\definecolor{scarletred1}{rgb}{0.937,0.161,0.161}

\definecolor{myblue}{HTML}{1e77b4}
\definecolor{myorange}{HTML}{ff7f0f}

\definecolor{mycolor}{rgb}{0.122, 0.435, 0.698}

\usepackage{cleveref}

\crefname{assumption}{assumption}{assumptions}
\Crefname{assumption}{Assumption}{Assumptions}
\newtheorem{theorem}{Theorem}[section]
\newtheorem{lemma}[theorem]{Lemma}
\newtheorem{definition}[theorem]{Definition}

\newtheorem{corollary}[theorem]{Corollary}
\newtheorem{assumption}[theorem]{Assumption}
\newtheorem{remark}[theorem]{Remark}
\newcounter{rmnum}
\newenvironment{romannum}
               {\begin{list}{{\upshape (\roman{rmnum})}}{\usecounter{rmnum}
                \setlength{\leftmargin}{0pt}
                \setlength{\itemindent}{42pt}}}{\end{list}}
\newcounter{muni}
\newenvironment{remunerate}
               {\begin{list}{{\upshape \arabic{muni}.}}{\usecounter{muni}
                \setlength{\leftmargin}{0pt}
                \setlength{\itemindent}{38pt}}}{\end{list}}

\allowdisplaybreaks

\title{Well-posedness of Bayesian Inverse Problems \\for Hyperbolic Conservation Laws}

\author{Siddhartha Mishra\footnote{Seminar for Applied Mathematics (SAM), ETH Z\"urich, R\"amistrasse 101, Z\"urich, Switzerland (\texttt{siddhartha.mishra@sam.math.ethz.ch} and \texttt{adrian.ruf@sam.math.ethz.ch})} , David Ochsner\footnote{Department of Mathematics, ETH Z\"urich, R\"amistrasse 101, Z\"urich, Switzerland (\texttt{david.ochsner@math.ethz.ch})} , Adrian\,M.~Ruf\footnotemark[1] , Franziska Weber\footnote{Department of Mathematical Sciences, Carnegie Mellon University, Pittsburgh, USA (\texttt{franzisw@andrew.cmu.edu})\newline
The research of Siddhartha Mishra and Adrian\,M.~Ruf is partially supported by the European Research Council Consolidator grant ERC-COG 770880 COMANFLO. Franziska Weber is partially supported by NSF DMS 1912854 and NSF OIA-DMR 2021019.}}

\begin{document}

\maketitle

\begin{abstract}
We study the well-posedness of the Bayesian inverse problem for scalar hyperbolic conservation laws where the statistical information about inputs such as the initial datum and (possibly discontinuous) flux function are inferred from noisy measurements. In particular, the Lipschitz continuity of the measurement to posterior map as well as the stability of the posterior to approximations, are established with respect to the Wasserstein distance. Numerical experiments are presented to illustrate the derived estimates.
\end{abstract}
%
%
%
%
%
%
%
%


\section{Introduction}
Hyperbolic systems of conservation laws are a large class of nonlinear PDEs which model a wide variety of phenomena in the sciences and engineering. The generic form of these PDEs is given by~\cite{dafermos2010hyperbolic}, 
\begin{equation}
\label{eq:syscl}
\begin{aligned}
    w_t + \nabla_x\cdot f(w) =0,& &&(x,t)\in \R^d \times (0,T),\\
    w(x,0)=\bar{w}(x),& &&x\in \R^d.
\end{aligned}
\end{equation}
Here, the solution field $w\colon \R^d \times [0,T] \to \R^m$ is the vector of conserved variables, $f\colon\R^m \to \R^m$ is the so-called \emph{flux function} and $\bar{w}$ is the initial datum.

Prototypical examples of systems of conservation laws include the compressible Euler equations of fluid dynamics, the shallow-water equations of oceanography, the MHD equations of plasma physics and the equations of nonlinear elasticity. The simplest examples are the so-called \emph{scalar conservation laws} i.e., \eqref{eq:syscl} with $m=1$, with the well-known Burgers' equation being a prototype. 

It is well-known that solutions of even scalar conservation laws develop discontinuities, such as \emph{shock waves}, for smooth initial data. Thus, the solutions of hyperbolic conservation laws are sought in the sense of distributions. However, these \emph{weak solutions} are supplemented with additional admissibility criteria or \emph{entropy conditions} to recover uniqueness \cite{dafermos2010hyperbolic}. 

The most studied aspect of PDEs such as hyperbolic conservation laws is the so-called \emph{forward problem} i.e., given the inputs (initial datum and flux function) $u = (\bar{w},f)$, find the \emph{entropy solution} $w$ of \eqref{eq:syscl}. Often, one is not necessarily interested in the whole solution field $w$ of \eqref{eq:syscl}, but rather in \emph{observables} or \emph{quantities of interest} of the solution. Hence, the forward problem reduces to an evaluation of the mapping $\mathcal{G}$, 
\begin{equation*}
    \mathcal{G}\colon X \to Y,\, u \mapsto y=\mathcal{G}(u),
\end{equation*}
which maps inputs $u \in X$ into observables $\mathcal{G}(u)\in Y$ of the solution, with $X,Y$ being suitable Banach spaces. 

However in practice, the inputs $u$ (which correspond to the initial datum and flux function in the context of hyperbolic conservation laws \eqref{eq:syscl}) may not be known exactly. Rather, one has to \emph{infer} them from \emph{measurements} of the observables. Hence, one is often interested in the so-called \emph{Inverse problem}, which amounts to finding information about the inputs $u$, given \emph{noisy measurements} of the form;
\begin{equation*}
    y=\mathcal{G}(u)+\eta,
\end{equation*}
with $\eta$ being a random variable encoding measurement noise. 

It is well-known that the \emph{deterministic} version of the inverse problem may be \emph{ill-posed} \cite{stuart2010inverse}. Although \emph{regularization} procedures have been widely developed in the last few decades to address this ill-posedness of the deterministic inverse problem, it is a widely held view that \emph{statistical approaches} might be better suited in this context. 
A very popular statistical approach \cite{stuart2010inverse} models the prior knowledge about the inputs $u$ in terms of a \emph{prior probability measure} $\mu_0 \in {\rm Prob}(X)$. Then the famous \emph{Bayes' theorem} can be used to update our knowledge of the inputs $u$ (and consequently the solution $w$) in terms of a \emph{posterior measure} $\mu^y \in {\rm Prob}(X)$, conditioned on the noisy measurements $y \in Y$. The posterior measure is given by the following expression of its Radon-Nikodym derivative,
\begin{equation}
    \frac{\diff\mu^y}{\diff \mu_0} (u) = \frac{1}{Z(y)} \exp(-\Phi(u;y)), \quad  Z(y)=\int_X \exp(-\Phi(u;y))\diff\mu_0(u)
    \label{eqn: Radon--Nikodym derivative}
\end{equation}
Here, $\Phi$ is the \emph{log-likelihood} with respect to the measurements $y$. 

It is to be noted that the Bayesian formulation encodes a regularized version of the underlying deterministic problem as the latter is a \emph{maximum a posteriori} (MAP) estimator of the former, with a suitable choice of the prior \cite{stuart2010inverse}. 

The well-posedness of the Bayesian inverse problem refers to the rigorous demonstration of \emph{existence and uniqueness} of the posterior measure $\mu^y$, its continuous dependence and stability with respect to perturbations of the measurements $y$. Moreover, in practice, one \emph{approximates} the posterior computationally, for instance, by sampling from it with a Metropolis--Hastings-type Markov chain Monte Carlo (MCMC) algorithm. This in turn requires one to evaluate the likelihood in terms of \emph{numerical approximations} $\mathcal{G}^{\Delta}$ of the forward operator $\mathcal{G}$. Here, $\Delta$ is a numerical regularization parameter such as the mesh size or the time step. The accuracy of the resulting \emph{approximate posterior} $\mu^{y,\Delta}$ is also of great interest. 

The well-posedness of the Bayesian inverse problems has been studied extensively in recent years and is nicely summarized in \cite{stuart2010inverse}. It has been clearly established that the Bayesian inverse problem is well-posed as long as the forward map $\mathcal{G}$ is \emph{Lipschitz continuous}, with respect to suitable topologies. Even weaker assumptions on the forward map have been investigated recently in \cite{latz2020well,Sprungk2020}. 

Furthermore, these abstract assumptions on well-posedness have been verified and illustrated for a variety of elliptic, parabolic and linear hyperbolic PDEs, see \cite{stuart2010inverse} and references therein. The application of this theory to nonlinear hyperbolic PDEs, such as hyperbolic conservation laws \eqref{eq:syscl} is currently not available, except in \cite{herrmann2020deep} where the authors study an example of a scalar conservation law with uncertain flux. 

Given this context, our main goal in this paper is to study and establish well-posedness of the Bayesian inverse problem for hyperbolic conservation laws. We will focus on the scalar case ($m=1$ in \eqref{eq:syscl}) as no rigorous well-posedness results are available for the forward problem for systems of conservation laws, particularly in several space dimensions. 

To this end, we will also study the Lipschitz continuity of the posterior measure with respect to measurements in the \emph{Wasserstein distance} on probability measures. We observe that the standard framework for Bayesian inverse problems \cite{stuart2010inverse} uses the \emph{Hellinger distance} to investigate stability with respect to perturbations. However, the Wasserstein distance offers some advantages over the Hellinger distance. To illustrate this, consider two measures which are absolutely continuous with respect to a Gaussian reference measure (e.g., the prior measure $\mu_0$) and which are a distance $\eps$ apart in the Wasserstein distance. Then the difference between the means of the two measures is bounded by $\eps$ (cf. \Cref{rem: difference in moments}). In contrast, if the two measures are a distance $\eps$ apart in the Hellinger metric then the difference in the means is only bounded by $C\eps$ (see \cite[Lem.~6.37]{stuart2010inverse}) where the constant $C$ depends on the second moments and, in particular, can be arbitrarily large. Thus by bounding the Wasserstein distance, we can more effectively control the change in the posterior, caused either by perturbations of the measurement or by replacing the underlying forward map with a (numerical) approximation. 

We apply these abstract stability results to establish the well-posedness of the Bayesian inverse problem, for inferring initial data as well as flux functions of scalar conservation laws, from measurements. Moreover, we extend the results to cover the Bayesian inverse problem for a conservation law, corresponding to a flux function that can vary discontinuously in the space variable. Thus, we establish the first rigorous well-posedness results for Bayesian inverse problems for these nonlinear hyperbolic PDEs. 

The remainder of this paper is organized as follows: In \Cref{sec: well-posedness} we provide the general well-posedness theory in the spirit of \cite{stuart2010inverse}, but employing the Wasserstein distance instead of the Hellinger distance. \Cref{sec: approximation} contains general approximation results for posterior measures given by \eqref{eqn: Radon--Nikodym derivative}. We study inverse problems for multi-dimensional scalar conservation laws and one-dimensional scalar conservation laws with discontinuous flux in \Cref{sec: conservation laws}. Lastly, in \Cref{sec: numerical experiments} we present a series of numerical experiments illustrating, in particular, the convergence of the approximated posterior distribution under refinement of the finite-dimensional approximation.



\section{Well-posedness of general Bayesian inverse problems in the Wasserstein distance}\label{sec: well-posedness}

The probability measure of interest is defined through a density with respect to a prior reference measure $\mu_0$ which, by shift of origin, we take to have mean zero. Further, we assume that this prior measure is Gaussian with covariance operator $\mathcal{C}$. We write $\mu_0=\mathcal{N}(0,\mathcal{C})$.

\begin{assumption}\label{Assumption on phi}
    For some separable Banach space $X$ with $\mu_0(X)=1$, the function $\Phi\colon X\times Y\to\R$ satisfies the following:
    \begin{romannum}
        \item\label{item: Assumption on phi - lower bound} for every $\eps>0$ and $r>0$ there is $M=M(\eps,r)\in\R$ such that for all $u\in X$ and $y\in Y$ with $\norm{y}_Y<r$
        \begin{equation*}
            \Phi(u;y) \geq M - \eps\norm{u}_X^2;
        \end{equation*}
        \item\label{item: Assumption on phi - upper bound} for every $r>0$ there is a $K=K(r)>0$ such that for all $u\in X$ and $y\in Y$ with $\norm{u}_X,\norm{y}_Y<r$
        \begin{equation*}
            \Phi(u;y) \leq K;
        \end{equation*}
        \item\label{item: Assumption on phi - Lipschitz continuity in u} for every $r>0$ there exists $L=L(r)>0$ such that for all $u,u'\in X$ and $y\in Y$ with $\norm{u}_X,\norm{u'}_X,\norm{y}_Y<r$
        \begin{equation*}
            \abs{\Phi(u;y) - \Phi(u';y)} \leq L\norm{u-u'}_X.
        \end{equation*}
        \item\label{item: Assumption on phi - Lipschitz continuity in y} for all $\eps>0$ and $r>0$ there is $C=C(\eps,r)\in\R$ such that for all $y,y'\in Y$ with $\norm{y},\norm{y'}<r$ and for all $u\in X$
        \begin{equation*}
            \abs{\Phi(u;y) - \Phi(u;y')} \leq \exp\left(\eps\norm{u}_X^2+C\right)\norm{y-y'}_Y.
        \end{equation*}
    \end{romannum}
\end{assumption}

Note that \Cref{Assumption on phi}~\eqref{item: Assumption on phi - lower bound} and~\eqref{item: Assumption on phi - upper bound} will lead to bounds on the normalization constant $Z$ from above and below. \Cref{Assumption on phi}~\eqref{item: Assumption on phi - Lipschitz continuity in u} and~\eqref{item: Assumption on phi - Lipschitz continuity in y} are Lipschitz conditions in $u$ and $y$ respectively.

For Bayesian inverse problems in which a finite number of observations are made and the observation error $\eta$ is mean zero Gaussian with covariance matrix $\Gamma$, the potential $\Phi$ has the form
\begin{equation}
    \Phi(u;y) = \frac{1}{2} \abs{y-\mathcal{G}(u)}_\Gamma^2,
    \label{eqn: phi in the case of finitely many observations}
\end{equation}
where $y\in\R^m$ is the data, $\mathcal{G}\colon X\to\R^m$ is the observation operator, and $\abs{\cdot}_\Gamma$ is a covariance weighted norm on $\R^m$. In this case, we can translate \Cref{Assumption on phi} in terms of $\mathcal{G}$.
\begin{assumption}\label{Assumption on G}
    For some separable Banach space $X$ with $\mu_0(X)=1$, the function $\mathcal{G}\colon X\to\R^m$ satisfies the following:
    \begin{romannum}
        \item\label{item: Assumption on G - exponential bound} for every $\eps>0$ there is $M=M(\eps)\in\R$ such that for all $u\in X$
        \begin{equation*}
            \abs{\mathcal{G}(u)}_\Gamma \leq \exp\left(\eps\norm{u}_X^2 +M\right);
        \end{equation*}
        \item\label{item: Assumption on G - Lipshitz continuity in u} for every $r>0$ there is a $K=K(r)>0$ such that for all $u,u'\in X$ with $\norm{u}_X,\norm{u'}_X<r$
        \begin{equation*}
            \abs{\mathcal{G}(u)-\mathcal{G}(u')}_\Gamma \leq K\norm{u-u'}_X.
        \end{equation*}
    \end{romannum}
\end{assumption}
\begin{lemma}[{\cite[Lem.~2.8]{stuart2010inverse}}]\label{lem: Assumption 2 implies Assumption 1}
    Assume that $\mathcal{G}\colon X\to\R^m$ satisfies \Cref{Assumption on G} and that $\mu_0$ is a Gaussian measure with $\mu_0(X)=1$. Then $\Phi\colon X\times \R^m\to\R$ given by~\eqref{eqn: phi in the case of finitely many observations} satisfies \Cref{Assumption on phi} with $(y,\norm{\cdot}_Y) = (\R^m,\abs{\cdot}_\Gamma)$.
    In particular, if $\mathcal{G}$ satisfies \Cref{Assumption on G}~\eqref{item: Assumption on G - exponential bound} then $\Phi$ given by~\eqref{eqn: phi in the case of finitely many observations} satisfies \Cref{Assumption on phi}~\eqref{item: Assumption on phi - lower bound},\eqref{item: Assumption on phi - upper bound}, and \eqref{item: Assumption on phi - Lipschitz continuity in y}.
\end{lemma}
\begin{proof}
    \Cref{Assumption on phi}~\eqref{item: Assumption on phi - lower bound} is trivially satisfied since $\Phi$ is nonnegative. Let now $r>0$, $u,u'\in X$ and $y,y'\in\R^m$ all with norm less than $r$. Using the exponential bound on $\mathcal{G}$ (with $\eps=1$) we find
    \begin{align*}
        \Phi(u;y) &\leq \abs{y}_\Gamma^2 + \abs{\mathcal{G}(u)}_\Gamma^2\\
        &\leq r^2 + \exp\left(\norm{u}_X^2+M\right)\\
        &\leq r^2 + \exp(r^2 + M)
    \end{align*}
    which gives \Cref{Assumption on phi}~\eqref{item: Assumption on phi - upper bound}. \Cref{Assumption on phi}~\eqref{item: Assumption on phi - Lipschitz continuity in u} follows from \Cref{Assumption on G}~\eqref{item: Assumption on G - Lipshitz continuity in u} because
    \begin{align*}
        \abs{\Phi(u;y)-\Phi(u',y)} \leq \frac{1}{2} &\abs{2y-\mathcal{G}(u)-\mathcal{G}(u')}_\Gamma \abs{\mathcal{G}(u)-\mathcal{G}(u')}_\Gamma\\
        &\leq C\left( |y| + N\norm{u-u'}_X \right) K\norm{u-u'}_X\\
        &\leq C\left( |y| + \norm{u}_X + \norm{u'}_X \right) \norm{u-u'}_X\\
        &\leq C r \norm{u-u'}_X.
    \end{align*}
    Lastly, for $\eps>0$, using the exponential bound on $\mathcal{G}$ we get
    \begin{align*}
        \abs{\Phi(u;y) - \Phi(u,y')} &\leq \frac{1}{2}\abs{y+y'-2 \mathcal{G}(u)}_\Gamma \abs{y-y'}_\Gamma\\
        &\leq C \left( \abs{y} + \abs{y'} + \exp\left( \eps\norm{u}_X^2 + M \right) \right) \abs{y-y'}\\
        &\leq \exp\left( \eps\norm{u}_X^2 +C(\eps,r) \right)\abs{y-y'}.
    \end{align*}
\end{proof}

The following theorem is due to Stuart and shows that $\mu^y$ given by~\eqref{eqn: Radon--Nikodym derivative} is a well-defined probability measure provided $\Phi$ satisfies a Lipschitz condition in $u$.

\begin{theorem}[{\cite[Thm.~4.1]{stuart2010inverse}}]\label{thm: Stuart existence of posterior measure}
    Let $\Phi$ satisfy \Cref{Assumption on phi}~\eqref{item: Assumption on phi - lower bound}, \eqref{item: Assumption on phi - upper bound}, and~\eqref{item: Assumption on phi - Lipschitz continuity in u} and assume that $\mu_0$ is a Gaussian measure satisfying $\mu_0(X)=1$.
    Then $\mu^y$ given by~\eqref{eqn: Radon--Nikodym derivative} is a well-defined probability measure on $X$.
\end{theorem}

We have the following immediate corollary for Bayesian inverse problems with a finite number of observations and $\Phi$ of the form~\eqref{eqn: phi in the case of finitely many observations}.
\begin{corollary}
    Assume that $\Phi\colon X\times Y\to\R$ is given by~\eqref{eqn: phi in the case of finitely many observations} and let $\mathcal{G}$ satisfy \Cref{Assumption on G}. Let further $\mu_0$ be a Gaussian measure satisfying $\mu_0(X)=1$. Then $\mu^y$ given by~\eqref{eqn: Radon--Nikodym derivative} is a well-defined probability measure on $X$.
\end{corollary}

\subsection{Well-posedness in the Wasserstein distance}

The $1$-Wasserstein distance between two probability measures $\mu$ and $\mu'$ with finite first moments
\begin{equation*}
    \int_X \|u\|_X\diff\mu(u),\int_X \|u\|_X\diff\mu'(u)<\infty
\end{equation*}
is defined as
\begin{equation*}
    W_1(\mu,\mu') = \sup_{\stackrel{\psi\in\mathcal{C}_b(X)}{\norm{\psi}_{\text{Lip}}\leq 1}} \int_X \psi(u) \diff(\mu-\mu')(u),
\end{equation*}
see \cite{Villani}.
Note that by the Fernique Theorem all moments of $u$ in $X$ are finite under a Gaussian measure (cf. \Cref{thm: Fernique}).
\begin{remark}\label{rem: difference in moments}
    The difference between the first moments of two probability measures $\mu$ and $\mu'$ is bounded by the Wasserstein distance between those measures:
    \begin{align*}
        \left\|\int_X u\diff\mu(u)-\int_X u\diff\mu'(u)\right\|_X = \left\|\int_X u\diff(\mu-\mu')(u)\right\|_X \leq \int_X\|u\|_X\diff(\mu-\mu')(u) \leq W_1(\mu,\mu').
    \end{align*}
\end{remark}

We show that the posterior measure $\mu^y$ is Lipschitz continuous with respect to the data $y$ in the $1$-Wasserstein distance. This constitutes a well-posedness result for the posterior measure. The result, and proof, is similar to that in~\cite{stuart2010inverse} concerning well-posedness in the Hellinger distance.

\begin{theorem}[{Well-posedness in $W_1$}]\label{thm: Well-posedness in W_1}
    Let $\Phi$ satisfy \Cref{Assumption on phi}~\eqref{item: Assumption on phi - lower bound}, \eqref{item: Assumption on phi - upper bound}, and~\eqref{item: Assumption on phi - Lipschitz continuity in y}. Assume also that $\mu_0$ is a Gaussian measure satisfying $\mu_0(X)=1$ and that for all $y\in Y$ the measure $\mu^y$ is absolutely continuous with respect to $\mu_0$, $\mu^y\ll\mu_0$, with Randon--Nikod\'{y}m derivative given by~\eqref{eqn: Radon--Nikodym derivative}. Then $y\mapsto\mu^y$ is Lipschitz continuous with respect to the $1$-Wasserstein distance: if $\mu^y$ and $\mu^{y'}$ are two measures corresponding to data $y$ and $y'$ then for all $r>0$ there exists $C=C(r)>0$ such that, if $\norm{y}_Y,\norm{y'}_Y<r$, then
    \begin{equation*}
        W_1(\mu^y,\mu^{y'}) \leq C\norm{y-y'}_Y.
    \end{equation*}
\end{theorem}
\begin{proof}
    In the following, we will write $Z$ and $Z'$ for $Z(y)$ and $Z(y')$ respectively (where $Z$ is defined in~\eqref{eqn: Radon--Nikodym derivative}).
    From \Cref{Assumption on phi}~\eqref{item: Assumption on phi - upper bound} we get for any $r>0$ and $\norm{y}_Y<r$
    \begin{equation*}
        \abs{Z} \geq \int_{\{\norm{u}_X<r\}}\exp(-L)\diff\mu_0(u) \geq \exp(-L)\mu_0(\{\norm{u}_X<r\}).
    \end{equation*}
    This lower bound is positive since $\mu_0$ has full measure on $X$ and is Gaussian so that all balls in $X$ have positive probability. We have an analogous lower bound for $\abs{Z'}$.

    Using the estimate
    \begin{equation}
        \abs{\exp(a) - \exp(b)} \leq (\exp(a)\vee \exp(b)) |a-b|,
        \label{eqn: exp estimate}
    \end{equation}
    \Cref{Assumption on phi}~\eqref{item: Assumption on phi - lower bound}, \eqref{item: Assumption on phi - Lipschitz continuity in y} and the fact that $\mu_0$ is a Gaussian measure so that the Fernique \Cref{thm: Fernique} applies, we find for $\norm{y}_Y,\norm{y'}_Y<r$
    \begin{align*}
        \abs{Z-Z'} &\leq \int_X \left(\exp(-\Phi(u;y))\vee\exp(-\Phi(u;y'))\right)\abs{\Phi(u;y) - \Phi(u;y')} \diff\mu_0(u)\\
        &\leq C \int_X \exp\left(\eps\norm{u}_X^2-M\right)\exp\left(\eps\norm{u}_X^2 + C\right)\norm{y-y'}_Y \diff\mu_0(u)\\
        &= C\int_X \exp\left(2\eps\norm{u}_X^2\right)\diff\mu_0(u) \norm{y-y'}_Y\\
        &= C \norm{y-y'}_Y.
    \end{align*}

    Now, let $\psi\in\mathcal{C}_b(X)$ with $\norm{\psi}_{\text{Lip}}\leq 1$. Since $\mu^y$ and $\mu^{y'}$ are probability measures, we have
    \begin{align*}
        \int_X \psi(u)\diff(\mu^y -\mu^{y'})(u) &= \int_X (\psi(u)-\psi(0)) \diff(\mu^y -\mu^{y'})(u) + \int_X \psi(0)\diff(\mu^y -\mu^{y'})(u)\\
        &= \int_X (\psi(u) -\psi(0)) \diff(\mu^y -\mu^{y'})(u)\\
        &= \int_X (\psi(u) - \psi(0)) \left( \frac{\diff\mu^y}{\diff\mu_0}(u) - \frac{\diff\mu^{y'}}{\diff\mu_0}(u) \right) \diff\mu_0(u)\\
        &= \int_X (\psi(u) - \psi(0)) \left( Z^{-1}\exp(-\Phi(u;y)) - (Z')^{-1}\exp(-\Phi(u;y')) \right) \diff\mu_0(u)\\
        &= I_1 + I_2
    \end{align*}
    where
    \begin{align*}
        I_1 &= \int_X (\psi(u) - \psi(0)) Z^{-1} \left( \exp(-\Phi(u;y)) - \exp(-\Phi(u;y')) \right) \diff\mu_0(u),\\
        I_2 &= \int_X (\psi(u) - \psi(0)) \left(Z^{-1} - (Z')^{-1}\right) \exp(-\Phi(u;y'))\diff\mu_0(u) 
    \end{align*}
    Using again the estimate~\eqref{eqn: exp estimate},
    the fact that $\norm{\psi}_{\text{Lip}}\leq 1$, \Cref{Assumption on phi}~\eqref{item: Assumption on phi - lower bound}, and~\eqref{item: Assumption on phi - Lipschitz continuity in y} we obtain
    \begin{align*}
        Z I_1 &\leq \int_X \abs{\psi(u)-\psi(0)} \abs{\exp(-\Phi(u;y)) - \exp(-\Phi(u;y'))} \diff\mu_0(u)\\
        &\leq \int_X \norm{u}_X ( \exp(-\Phi(u;y)) \vee \exp(-\Phi(u;y')) )\abs{\Phi(u;y)) - \Phi(u;y')} \diff\mu_0(u)\\
        &\leq C \left(\int_X \norm{u}_X \exp\left(2\eps\norm{u}_X^2\right) \diff\mu_0(u)\right) \norm{y-y'}_Y.
    \end{align*}
    Since all moments of $u$ in $X$ are finite under the Gaussian measure $\mu_0$ by the Fernique Theorem, the integral in the last line can be bounded by using the Cauchy--Schwarz inequality and again the Fernique Theorem. Since $Z$ is bounded from below by a positive constant, this gives a bound on $I_1$.

    Using the fact that $\norm{\psi}_{\text{Lip}}\leq 1$, \Cref{Assumption on phi}~\eqref{item: Assumption on phi - lower bound}, the above bound on $\abs{Z-Z'}$ and again the fact that $Z$ and $Z'$ are bounded from below by a positive constant, we get
    \begin{align*}
        I_2 \leq &\int_X \norm{u}_X \abs{Z^{-1} - (Z')^{-1}} \exp(-\Phi(u;y')) \diff\mu_0(u)\\
        &\leq C \abs{Z^{-1}-(Z')^{-1}} \int_X \norm{u}_X \exp(\eps\norm{u}_X^2)\diff\mu_0(u)\\
        &\leq C\left( Z^{-2} \vee (Z')^{{-2}} \right) |Z-Z'|\\
        &\leq C \norm{y-y'}_Y.
    \end{align*}
    Here we used the same arguments as before to bound the integral $\int_X\norm{u}_X\exp\left(\eps\norm{u}_X^2\right)\diff\mu_0(u)$.
    Combining the bounds for $I_1$ and $I_2$ gives the desired continuity result in the Wasserstein distance.
\end{proof}

\begin{remark}
    In the proof of \Cref{thm: Well-posedness in W_1} we only use the assumption that $\mu_0$ is Gaussian to deduce that there exists $\alpha>0$ such that $\int_X \exp(\alpha\norm{u}_X^2)\diff\mu_0(u)<\infty$. Therefore, the statement of \Cref{thm: Well-posedness in W_1} readily extends to any prior measure $\mu_0$ with this property.
\end{remark}

For Bayesian inverse problems with finite data the potential has the form~\eqref{eqn: phi in the case of finitely many observations} where $y\in\R^m$ is the data $\mathcal{G}\colon X\to\R^m$ is the observation operator, and $\abs{\cdot}_\Gamma$ is a covariance weighted norm on $\R^m$.
By \Cref{lem: Assumption 2 implies Assumption 1} we know that \Cref{Assumption on G} implies \Cref{Assumption on phi} for $\Phi$ given by~\eqref{eqn: phi in the case of finitely many observations}. Thus, we have the following corollary of \Cref{thm: Well-posedness in W_1}.
\begin{corollary}\label{cor: finite data well-posedness}
    Assume that $\Phi\colon X\times Y\to \R$ is given by~\eqref{eqn: phi in the case of finitely many observations} and let $\mathcal{G}$ satisfy \Cref{Assumption on G}~\eqref{item: Assumption on G - exponential bound}. Assume further that $\mu_0$ is a Gaussian measure satisfying $\mu_0(X)=1$ and that for all $y\in Y$ the measure $\mu^y$ is absolutely continuous with respect to $\mu_0$, $\mu^y\ll \mu_0$, with Randon--Nikod\'{y}m derivative given by~\eqref{eqn: Radon--Nikodym derivative}. Then $y\mapsto\mu^y$ is Lipschitz continuous with respect to the $1$-Wasserstein distance: if $\mu^y$ and $\mu^{y'}$ are two measures corresponding to data $y$ and $y'$ then for all $r>0$ there exists $C=C(r)>0$ such that, if $\norm{y}_Y,\norm{y'}_Y<r$, then
    \begin{equation*}
        W_1(\mu^y,\mu^{y'}) \leq C\norm{y-y'}_Y
    \end{equation*}
\end{corollary}

\section{Approximation of posterior measures in the Wasserstein distance}\label{sec: approximation}
In order to implement algorithms designed to sample the posterior measure $\mu^y$, we need to make finite-dimensional approximations. Since the dependence on $y$ is not relevant in this section, we suppress it notationally and study measures $\mu$ given by
\begin{equation}
    \frac{\diff\mu}{\diff\mu_0}(u) = \frac{1}{Z}\exp(-\Phi(u))
    \label{eqn: RN derivative of mu}
\end{equation}
where the normalization constant $Z$ is given by
\begin{equation*}
    Z=\int_X\exp(-\Phi(u))\diff\mu_0(u).
\end{equation*}
We approximate $\mu$ by approximating $\Phi$ over some $N$-dimensional subspace of $X$. Specifically, we define $\mu^N$ by
\begin{equation}
    \frac{\diff\mu^N}{\diff\mu_0}(u) = \frac{1}{Z^N}\exp\left(-\Phi^N(u)\right)
    \label{eqn: RN derivative of mu^N}
\end{equation}
where
\begin{equation*}
    Z^N=\int_X\exp\left(-\Phi^N(u)\right)\diff\mu_0(u).
\end{equation*}
The following theorem bounds the $1$-Wasserstein distance between $\mu$ and $\mu^N$ in terms of the error in approximating $\Phi$. Note that this effectively translates approximation results for $\Phi$---which are determined by the forward problem---into approximation results for the posterior $\mu$.

\begin{theorem}\label{thm: Approximation of measures in W1}
    Assume that the measures $\mu$ and $\mu^N$ are both absolutely continuous with respect to $\mu_0$, satisfying $\mu_0(X)=1$, with Randon--Nikod\'{y}m derivative given by~\eqref{eqn: RN derivative of mu} and~\eqref{eqn: RN derivative of mu^N} and that $\Phi$ and $\Phi^N$ satisfy \Cref{Assumption on phi}~\eqref{item: Assumption on phi - lower bound} and~\eqref{item: Assumption on phi - upper bound} with constants uniform in $N$. Assume also that for any $\eps>0$ there is $K=K(\eps)>0$ such that
    \begin{equation}
        \abs{\Phi(u)-\Phi^N(u)} \leq K\exp\left(\eps\norm{u}_X^2\right) \Psi(N),
        \label{eqn: bound for the approximation error in phi}
    \end{equation}
    where $\Psi(N)\to 0$ as $N\to\infty$. Then the measures $\mu$ and $\mu^N$ are close with respect to the $1$-Wasserstein distance: there is a constant $C$, independent of $N$, such that
    \begin{equation*}
        W_1(\mu,\mu^N)\leq C \Psi(N).
    \end{equation*}
\end{theorem}
\begin{proof}
    The normalization constants $Z$ and $Z^N$ satisfy lower bounds independent of $N$ which are identical to that proved for $Z$ in the course of establishing \Cref{thm: Well-posedness in W_1}.

    Using the estimate~\eqref{eqn: exp estimate},
    \Cref{Assumption on phi}~\eqref{item: Assumption on phi - lower bound}, \eqref{eqn: bound for the approximation error in phi}, and the fact that $\mu_0$ is a Gaussian measure so that the Fernique \Cref{thm: Fernique} applies, we find
    \begin{align*}
        \abs{Z-Z^N} &\leq \int_X \left(\exp(-\Phi(u))\vee\exp(-\Phi^N(u))\right)\abs{\Phi(u) - \Phi^N(u)} \diff\mu_0(u)\\
        &\leq C \int_X \exp\left(\eps\norm{u}_X^2-M\right)\exp\left(\eps\norm{u}_X^2\right)K\Psi(N) \diff\mu_0(u)\\
        &= C\int_X \exp\left(2\eps\norm{u}_X^2\right)\diff\mu_0(u) \Psi(N)\\
        &= C \Psi(N).
    \end{align*}

    Now, let $\psi\in\mathcal{C}_b(X)$ with $\norm{\psi}_{\text{Lip}}\leq 1$. Since $\mu$ and $\mu^N$ are probability measures, we have
    \begin{align*}
        \int_X \psi(u)\diff\left(\mu -\mu^N\right)(u) &= \int_X (\psi(u)-\psi(0)) \diff\left(\mu -\mu^N\right)(u) + \int_X \psi(0)\diff\left(\mu -\mu^N\right)(u)\\
        &= \int_X (\psi(u) -\psi(0)) \diff\left(\mu -\mu^N\right)(u)\\
        &= \int_X (\psi(u) - \psi(0)) \left( \frac{\diff\mu}{\diff\mu_0}(u) - \frac{\diff\mu^N}{\diff\mu_0}(u) \right) \diff\mu_0(u)\\
        &= \int_X (\psi(u) - \psi(0)) \left( Z^{-1}\exp(-\Phi(u)) - \left(Z^N\right)^{-1}\exp(-\Phi^N(u)) \right) \diff\mu_0(u)\\
        &= I_1 + I_2
    \end{align*}
    where
    \begin{align*}
        I_1 &= \int_X (\psi(u) - \psi(0)) Z^{-1} \left( \exp(-\Phi(u)) - \exp(-\Phi^N(u)) \right) \diff\mu_0(u),\\
        I_2 &= \int_X (\psi(u) - \psi(0)) \left(Z^{-1} - \left(Z^N\right)^{-1}\right) \exp(-\Phi^N(u))\diff\mu_0(u).
    \end{align*}
    Using the estimate~\eqref{eqn: exp estimate} again as well as
    the fact that $\norm{\psi}_{\text{Lip}}\leq 1$, \Cref{Assumption on phi}~\eqref{item: Assumption on phi - lower bound}, and~\eqref{eqn: bound for the approximation error in phi} we obtain
    \begin{align*}
        Z I_1 &\leq \int_X \abs{\psi(u)-\psi(0)} \abs{\exp(-\Phi(u)) - \exp(-\Phi^N(y))} \diff\mu_0(u)\\
        &\leq \int_X \norm{u}_X ( \exp(-\Phi(u)) \vee \exp(-\Phi^N(u)) )\abs{\Phi(u)) - \Phi^N(u)} \diff\mu_0(u)\\
        &\leq C \left(\int_X \norm{u} \exp\left(2\eps\norm{u}_X^2\right) \diff\mu_0(u)\right) \Psi(N).
    \end{align*}
    Since all moments of $u$ in $X$ are finite under the Gaussian measure $\mu_0$ by the Fernique Theorem, the integral in the last line can be bounded by using the Cauchy--Schwarz inequality and again the Fernique Theorem. Since $Z$ is bounded from below by a positive constant, this gives a bound on $I_1$.

    Using the fact that $\norm{\psi}_{\text{Lip}}\leq 1$, \Cref{Assumption on phi}~\eqref{item: Assumption on phi - lower bound}, the above bound on $\abs{Z-Z^N}$ and again the fact that $Z$ and $Z^N$ are bounded from below by a positive constant independent of $N$, we get
    \begin{align*}
        I_2 \leq &\int_X \norm{u}_X \abs{Z^{-1} - (Z^N)^{-1}} \exp(-\Phi^N(u)) \diff\mu_0(u)\\
        &\leq C \abs{Z^{-1}-(Z^N)^{-1}} \int_X \norm{u}_X \exp(\eps\norm{u}_X^2)\diff\mu_0(u)\\
        &\leq C\left( Z^{-2} \vee (Z^N)^{{-2}} \right) \abs{Z-Z^N}\\
        &\leq C \Psi(N).
    \end{align*}
    Here we used the same arguments as before to bound the integral $\int_X\norm{u}_X\exp\left(\eps\norm{u}_X^2\right)\diff\mu_0(u)$.
    Combining the bounds for $I_1$ and $I_2$ gives the desired continuity result in the Wasserstein distance.
\end{proof}

Again, if the data is finite, the potential has the form~\eqref{eqn: phi in the case of finitely many observations}, where $y\in\R^m$ is the data, $\mathcal{G}\colon X\to\R^m$ is the observation operator, and $\abs{\cdot}_\Gamma$ is a covariance weighted norm on $\R^m$. If $\mathcal{G}^N$ is an approximation to $\mathcal{G}$ and we define
\begin{equation}
    \Phi^N(u;y)\coloneqq \abs{y-\mathcal{G}^N(u)}_\Gamma
    \label{eqn: Phi^N in the case of finitely many observations}
\end{equation}
then we can define an approximation $\mu^N$ to $\mu$ as in~\eqref{eqn: RN derivative of mu^N} and we have the following corollary.

\begin{corollary}\label{cor: finite data approximation}
    Assume that the measures $\mu$ and $\mu^N$ are both absolutely continuous with respect to $\mu_0$, satisfying $\mu_0(X)=1$, with Randon--Nikod\'{y}m derivative given by~\eqref{eqn: RN derivative of mu},~\eqref{eqn: phi in the case of finitely many observations} and~\eqref{eqn: RN derivative of mu^N},~\eqref{eqn: Phi^N in the case of finitely many observations} respectively. Assume also that $\mathcal{G}$ is approximated by a function $\mathcal{G}^N$ with the property that for any $\eps>0$ there is $K'=K'(\eps)>0$ such that
    \begin{equation}
        \abs{\mathcal{G}(u)-\mathcal{G}^N(u)}\leq K'\exp\left(\eps\norm{u}_X^2\right) \Psi(N),
        \label{eqn: bound for the approximation error in G}
    \end{equation}
    where $\Psi(N)\to 0$ as $N\to\infty$. If $\mathcal{G}$ and $\mathcal{G}^N$ satisfy \Cref{Assumption on G}~\eqref{item: Assumption on G - exponential bound} uniformly in $N$, then the measures $\mu$ and $\mu^N$ are close with respect to the $1$-Wasserstein distance: there is a constant $C$, independent of $N$, such that
    \begin{equation*}
        W_1(\mu,\mu^N)\leq C \Psi(N).
    \end{equation*}
\end{corollary}
\begin{proof}
    Using \Cref{Assumption on G}~\eqref{item: Assumption on G - exponential bound} we get for all $\eps>0$ and $y\in\R^m$
    \begin{align*}
        \abs{\Phi(u)-\Phi^N(u)} &\leq \frac{1}{2} \abs{2y - \mathcal{G}(u)-\mathcal{G}^N(u)}_\Gamma \abs{\mathcal{G}(u)-\mathcal{G}^N(u)}_\Gamma\\
        &\leq  C\left( \abs{y} + \exp\left(\eps\norm{u}_X^2 +M\right) \right) \exp\left(\eps\norm{u}_X^2\right)\Psi(N)\\
        &\leq C(2\eps,y) \exp\left(2\eps\norm{u}_X^2\right)\Psi(N)
    \end{align*}
    such that~\eqref{eqn: bound for the approximation error in phi} holds and, in view of \Cref{lem: Assumption 2 implies Assumption 1}, we can apply \Cref{thm: Approximation of measures in W1}.
\end{proof}

In \Cref{thm: Approximation of measures in W1} it is necessary that the constant in the error bound~\eqref{eqn: bound for the approximation error in phi} for approximating the function $\Phi$ by $\Phi^N$ is integrable by use of the Fernique \Cref{thm: Fernique}. In case such integrability is not at hand, we can still derive the convergence result, albeit at possibly weaker rates.

\begin{theorem}\label{thm: Approximation of measures in W1 without Fernique integrable constant}
    Assume that the measures $\mu$ and $\mu^N$ are both absolutely continuous with respect to $\mu_0$, satisfying $\mu_0(X)=1$, with Randon--Nikod\'{y}m derivative given by~\eqref{eqn: RN derivative of mu} and~\eqref{eqn: RN derivative of mu^N} and that $\Phi$ and $\Phi^N$ satisfy \Cref{Assumption on phi}~\eqref{item: Assumption on phi - lower bound} and~\eqref{item: Assumption on phi - upper bound} with constants uniform in $N$. Assume also that for any $R>0$ there is $K=K(R)>0$ such that for all $u\in X$ with $\norm{u}_X\leq R$
    \begin{equation}
        \abs{\Phi(u)-\Phi^N(u)} \leq K \Psi(N),
        \label{eqn: bound for the approximation error in phi without Fernique integrable constant}
    \end{equation}
    where $\Psi(N)\to 0$ as $N\to\infty$. Then 
    \begin{equation*}
        W_1(\mu,\mu^N) \to 0
    \end{equation*}
    as $N\to\infty$.
\end{theorem}
\begin{proof}
    The normalization constants $Z$ and $Z^N$ satisfy lower bounds independent of $N$ which are identical to that proved for $Z$ in the course of establishing \Cref{thm: Well-posedness in W_1}.

    Using the estimate~\eqref{eqn: exp estimate},
    \Cref{Assumption on phi}~\eqref{item: Assumption on phi - lower bound}, and \eqref{eqn: bound for the approximation error in phi without Fernique integrable constant}, we find
    \begin{align*}
        \abs{Z-Z^N} &\leq \int_X \abs{\exp(-\Phi(u)) - \exp(-\Phi^N(u))} \diff\mu_0(u)\\
        &\leq \int_{\{\norm{u}_X\leq R\}} \exp(\eps\norm{u}_X^2 -M) \abs{\Phi(u)-\Phi^N(u)}\diff\mu_0(u)\\
        &\phantom{\mathrel{=}}+ \int_{\{\norm{u}_X>R\}} 2\exp(\eps\norm{u}_X^2-M)\diff\mu_0(u)\\
        &\leq \exp(\eps R^2 -M) K(R)\Psi(N) + J_R\\
        &\eqqcolon K_1(R)\Psi(N) + J_R
    \end{align*}
    where
    \begin{equation*}
        J_R = \int_{\{\norm{u}_X>R\}} 2\exp(\eps\norm{u}_X^2-M)\diff\mu_0(u).
    \end{equation*}
    Because of the Fernique \Cref{thm: Fernique}, $J_R\to 0$ as $R\to\infty$. Therefore, for any $\delta>0$ we can choose $R$ sufficiently large such that $J_R <\delta$. By choosing $N$ large enough that $K_1(R)\Psi(N) <\delta$, we get $\abs{Z-Z^N}<2\delta$. Therefore, we have $Z^N\to Z$ as $N\to\infty$.

    Now, let $\psi\in\mathcal{C}_b(X)$ with $\norm{\psi}_{\text{Lip}}\leq 1$. Since $\mu$ and $\mu^N$ are probability measures, we have
    \begin{align*}
        \int_X \psi(u)\diff\left(\mu -\mu^N\right)(u) &= \int_X (\psi(u)-\psi(0)) \diff\left(\mu -\mu^N\right)(u) + \int_X \psi(0)\diff\left(\mu -\mu^N\right)(u)\\
        &= \int_X (\psi(u) -\psi(0)) \diff\left(\mu -\mu^N\right)(u)\\
        &= \int_X (\psi(u) - \psi(0)) \left( \frac{\diff\mu}{\diff\mu_0}(u) - \frac{\diff\mu^N}{\diff\mu_0}(u) \right) \diff\mu_0(u)\\
        &= \int_X (\psi(u) - \psi(0)) \left( Z^{-1}\exp(-\Phi(u)) - \left(Z^N\right)^{-1}\exp(-\Phi^N(u)) \right) \diff\mu_0(u)\\
        &= I_1 + I_2
    \end{align*}
    where
    \begin{align*}
        I_1 &= \int_X (\psi(u) - \psi(0)) Z^{-1} \left( \exp(-\Phi(u)) - \exp(-\Phi^N(u)) \right) \diff\mu_0(u),\\
        I_2 &= \int_X (\psi(u) - \psi(0)) \left(Z^{-1} - \left(Z^N\right)^{-1}\right) \exp(-\Phi^N(u))\diff\mu_0(u).
    \end{align*}
    Using the estimate~\eqref{eqn: exp estimate},
    the fact that $\norm{\psi}_{\text{Lip}}\leq 1$, \Cref{Assumption on phi}~\eqref{item: Assumption on phi - lower bound}, and~\eqref{eqn: bound for the approximation error in phi without Fernique integrable constant} we obtain
    \begin{align*}
        I_1 &\leq Z^{-1} \int_X \abs{\psi(u)-\psi(0)} \abs{\exp(-\Phi(u)) - \exp(-\Phi^N(y))} \diff\mu_0(u)\\
        &\leq Z^{-1}\int_{\{\norm{u}_X\leq R\}} \norm{u}_X \exp(\eps\norm{u}_X^2 -M) K(R)\Psi(N) \diff\mu_0(u)\\
        &\phantom{\mathrel{=}} + Z^{-1} \int_{\{\norm{u}_X>R\}} 2\norm{u}_X \exp(\eps\norm{u}_X^2 -M)\diff\mu_0(u)\\
        &\leq  Z^{-1} R \exp(\eps R^2-M)K(R)\Psi(N) + Z^{-1}\int_{\{\norm{u}_X>R\}} 2\norm{u}_X \exp(\eps\norm{u}_X^2 -M)\diff\mu_0(u).
    \end{align*}
    Since $Z$ is bounded from below and all momemts of $u$ in $X$ are finite under the Gaussian measure $\mu_0$, an argument similar to the one above for $\abs{Z-Z^N}$ shows that $I_1\to 0$ as $N\to\infty$.

    Using the fact that $\norm{\psi}_{\text{Lip}}\leq 1$, \Cref{Assumption on phi}~\eqref{item: Assumption on phi - lower bound}, the bound on $\abs{Z-Z^N}$ and again the fact that $Z$ and $Z^N$ are bounded from below by a positive constant independent of $N$, we get as before
    \begin{align*}
        I_2 \leq &\int_X \norm{u}_X \abs{Z^{-1} - (Z^N)^{-1}} \exp(-\Phi^N(u)) \diff\mu_0(u)\\
        &\leq C \abs{Z^{-1}-(Z^N)^{-1}} \int_X \norm{u}_X \exp(\eps\norm{u}_X^2)\diff\mu_0(u)\\
        &\leq C\left( Z^{-2} \vee (Z^N)^{{-2}} \right) |Z-Z^N|.
    \end{align*}
    Thus, we have $I_2\to 0$ as $N\to\infty$.
    Combining gives the claimed continuity result in the Wasserstein distance.
\end{proof}

If the data is finite, we can derive the following Corollary in analogy to \Cref{cor: finite data approximation}.
\begin{corollary}
    Assume that the measures $\mu$ and $\mu^N$ are both absolutely continuous with respect to $\mu_0$, satisfying $\mu_0(X)=1$, with Randon--Nikod\'{y}m derivative given by~\eqref{eqn: RN derivative of mu},~\eqref{eqn: phi in the case of finitely many observations} and~\eqref{eqn: RN derivative of mu^N},~\eqref{eqn: Phi^N in the case of finitely many observations} respectively. Assume also that $\mathcal{G}$ is approximated by a function $\mathcal{G}^N$ with the property that for any $R>0$ there is $K'=K'(R)>0$ such that for all $u\in X$ with $\norm{u}_X\leq R$
    \begin{equation*}
        \abs{\mathcal{G}(u)-\mathcal{G}^N(u)}\leq K' \Psi(N),
    \end{equation*}
    where $\Psi(N)\to 0$ as $N\to\infty$. If $\mathcal{G}$ and $\mathcal{G}^N$ satisfy \Cref{Assumption on G}~\eqref{item: Assumption on G - exponential bound} uniformly in $N$, then 
    \begin{equation*}
        W_1(\mu,\mu^N)\to 0
    \end{equation*}
    as $N\to\infty$.
\end{corollary}
\begin{proof}
    Using \Cref{Assumption on G}~\eqref{item: Assumption on G - exponential bound} (with $\eps=1$) we get for all $R>0$, $u\in X$ with $\norm{u}_X\leq R$, and $y\in\R^m$
    \begin{align*}
        \abs{\Phi(u)-\Phi^N(u)} &\leq \frac{1}{2} \abs{2y - \mathcal{G}(u)-\mathcal{G}^N(u)}_\Gamma \abs{\mathcal{G}(u)-\mathcal{G}^N(u)}_\Gamma\\
        &\leq  C\left( \abs{y} + \exp\left(\norm{u}_X^2 +M(1)\right) \right) K(R)\Psi(N)\\
        &\leq C \left(\abs{y} + \exp(R^2+M(1))\right) K(R)\Psi(N)
    \end{align*}
    such that~\eqref{eqn: bound for the approximation error in phi without Fernique integrable constant} holds and, in view of \Cref{lem: Assumption 2 implies Assumption 1}, we can apply \Cref{thm: Approximation of measures in W1 without Fernique integrable constant}.
\end{proof}

\section{Bayesian inverse problems for conservation laws}\label{sec: conservation laws}

In this section we use stability and convergence rate estimates for scalar conservation laws to establish that the associated inverse problems may be placed in the general framework for Bayesian inverse problems in the Wasserstein distance. To this end, we consider scalar conservation laws for which the available theory is very mature as well as scalar conservation laws with discontinuous flux where stability in the model parameters and convergence rates were established only very recently.
We start by recalling the necessary well-posedness results for entropy solutions of scalar conservation laws.

\subsection{Scalar conservation laws in several space dimensions}

We consider the Cauchy problem for scalar conservation laws of the form
\begin{gather}
    \begin{aligned}
        w_t + \nabla_x\cdot f(w) =0,& &&(x,t)\in \R^d\times(0,T),\\
        w(x,0) = \bar{w}(x),& &&x\in \R^d.
    \end{aligned}
    \label{eqn: conservation law}
\end{gather}
Here, the unknown is $w\colon \R^d\times [0,T]\to\R$ and $f=(f_1,\ldots,f_d)\in\mathcal{C}^{0,1}(\R;\R^d)$ is the flux function.

\subsubsection{Entropy solutions}
Since weak solutions of~\eqref{eqn: conservation law} are not unique we consider entropy solutions in the following sense.
\begin{definition}
    We call a function $w\in\mathrm{L}^\infty(\R^d\times(0,T))\cap\mathcal{C}([0,T];\mathrm{L}^1(\R^d))$ an entropy solution of~\eqref{eqn: conservation law} if for all $c\in\R$
    \begin{equation*}
        \int_0^T\int_{\R^d} \bigg( \abs{w-c}\varphi_t + \sign(w-c)\sum_{j=1}^d(f_j(u) - f_j(c))\varphi_{x_j} \bigg)\diff x\diff t
         + \int_{\R^d} \abs{\bar{w}(x)-c}\varphi(x,0)\diff x \geq 0
    \end{equation*}
    for all nonnegative $\varphi\in\mathcal{C}_c^\infty(\R^d\times[0,T))$.
\end{definition}
It is well-known that the Cauchy problem~\eqref{eqn: conservation law} admits, for each $\bar{w}\in\left(\mathrm{L}^1\cap\mathrm{BV}\right)(\R^d)$, a unique entropy solution and we summarize the classical results on existence and uniqueness of entropy solutions in the following theorem (see, e.g., \cite{godlewski1991hyperbolic}).
\begin{theorem}
\phantom{.}
    \begin{remunerate}
        \item For every $\bar{w}\in\mathrm{L}^\infty(\R^d)$, \eqref{eqn: conservation law} admits a unique entropy solution $w\in\mathrm{L}^\infty(\R^d\times(0,T))$.
        \item For every $t>0$, the solution operator $S_t$ given by
        \begin{equation*}
            S_t \bar{w} = w(\cdot,t)
        \end{equation*}
        satisfies
        \begin{romannum}
            \item $S_t\colon\mathrm{L}^1(\R^d)\to\mathrm{L}^1(\R^d)$ is a contraction, i.e.,
            \begin{equation*}
                \norm{S_t \bar{w} - S_t \hat{w}}_{\mathrm{L}^1(\R^d)} \leq \norm{\bar{w}-\hat{w}}_{\mathrm{L}^1(\R^d)}
            \end{equation*}
            for all $\bar{w},\hat{w}\in\mathrm{L}^1(\R^d)$.
            \item $S_t$ maps $(\mathrm{L}^1\cap\mathrm{BV})(\R^d)$ into itself and
            \begin{equation*}
                \mathrm{TV}(S_t \bar{w}) \leq \mathrm{TV}(\bar{w})
            \end{equation*}
            for all $\bar{w}\in (\mathrm{L}^1\cap\mathrm{BV})(\R^d)$.
            \item For every $\bar{w}\in (\mathrm{L}^1\cap\mathrm{L}^\infty)(\R^d)$
            \begin{align}
                \norm{S_t \bar{w}}_{\mathrm{L}^1(\R^d)} &\leq \norm{\bar{w}}_{\mathrm{L}^1(\R^d)}, \label{eqn: L^1 bound of the ES}\\
                \norm{S_t \bar{w}}_{\mathrm{L}^\infty(\R^d)} &\leq \norm{\bar{w}}_{\mathrm{L}^\infty(\R^d)}. \label{eqn: L^infty bound of the ES}
            \end{align}
            \item $t\mapsto S_t$ is a uniformly continuous mapping from $\mathrm{L}^1(\R^d)$ into $\mathcal{C}_b([0,\infty);\mathrm{L}^1(\R^d))$ and
            \begin{equation*}
                \norm{t\mapsto S_t \bar{w}}_{\mathcal{C}([0,T];\mathrm{L}^1(\R^d))} \leq \norm{\bar{w}}_{\mathrm{L}^1(\R^d)}
            \end{equation*}
            for all $\bar{w}\in\mathrm{L}^1(\R^d)$.
        \end{romannum}
    \end{remunerate}
\end{theorem}

In the following, our notation for the solution operator will not only carry the dependence on the initial datum, but also on the flux. We will write
\begin{equation*}
    S_t (\bar{w},f) = w(\cdot,t)
\end{equation*}
and understand $S_t$ as a map from $\left(\mathrm{L}^1\cap\mathrm{L}^\infty\cap\mathrm{BV}\right)(\R^d)\times\mathcal{C}^{0,1}(\R;\R^d)$ to $\mathrm{L}^1(\R^d)$ with the properties listed above. The following theorem shows that this map is locally Lipschitz continuous.
\begin{theorem}[{\cite[Thm.~4.3]{HoldenEtAl2015}}]
    Assume $\bar{w},\hat{w}\in\left(\mathrm{L}^1\cap\mathrm{L}^\infty\cap\mathrm{BV}\right)(\R^d)$ and $f,g\in\mathcal{C}^{0,1}(\R;\R^d)$. Then the solution operator satisfies
    \begin{equation}
        \norm{S_t(\bar{w},f) - S_t(\hat{w},g)}_{\mathrm{L}^1(\R^d)} \leq \norm{\bar{w}-\hat{w}}_{\mathrm{L}^1(\R^d)} + t \min\left(\mathrm{TV}(\bar{w}),\mathrm{TV}(\hat{w})\right) \norm{f-g}_{\mathrm{Lip}}
        \label{eqn: Lipschitz continuity of the solution operator (in f and u_0)}
    \end{equation}
    for every $0\leq t\leq T$.
\end{theorem}


\subsubsection{Finite volume methods}

We briefly describe the conventional approach of numerically approximating solutions of scalar conservation laws through finite volume methods (cf.~\cite{leveque1992numerical,crandall1980monotone,kuznetsov76}).

We discretize the spatial computational domain with cells
\begin{equation*}
    \cell_{i_1,\ldots,i_d} \coloneqq (x^1_{i_1-\hf},x^1_{i_1+\hf})\times\ldots\times(x^d_{i_d-\hf},x^d_{i_d+\hf}) \subset\R^d
\end{equation*}
with corresponding cell midpoints
\begin{equation*}
    x_{i_1,\ldots,i_d} \coloneqq \left( \frac{x^1_{i_1+\hf} + x^1_{i_1-\hf}}{2},\ldots,\frac{x^d_{i_d+\hf} + x^d_{i_d-\hf}}{2} \right).
\end{equation*}
For simplicity, we assume that the mesh is equidistant, meaning
\begin{equation*}
    x^k_{i_k+\hf} - x^k_{i_k-\hf} = \Dx, \qquad\text{for all }k=1,\ldots,d\text{ and } i_k\in\Z,
\end{equation*}
for some $\Dx>0$. We consider a uniform discretization in time with time step $\Dt>0$ such that the time interval $[0,T]$ is partitioned into intervals $[t^n,t^{n+1})$ where $t^n=n\Dt$ and that $\lambda\coloneqq \frac{\Dt}{\Dx}$ is constant and satisfies a standard CFL condition based on the maximum wave speed (see e.g.~\cite{godlewski1991hyperbolic}).

We consider the following numerical scheme:
\begin{gather}
    \begin{aligned}
        w_{i_1,\ldots,i_d}^{n+1} = w_{i_1,\ldots,i_d}^n - \lambda\sum_{k=1}^d \left( F_{i_1,\ldots,i_k+\hf,\ldots,i_d}^{k,n} - F_{i_1,\ldots,i_k-\hf,\ldots,i_d}^{k,n} \right),&\\
        w_{i_1,\ldots,i_d}^0 = \frac{1}{\Dx^d} \int_{\cell_{i_1,\ldots,i_d}} \bar{w}(x)\diff x,&
    \end{aligned}
    \label{eqn: Finite volume method}
\end{gather}
where $F^{k,n}$ is a numerical flux function in direction $k$. In a $(2p+1)$-point scheme, the numerical flux function $F_{i_1,\ldots,i_k+\hf,\ldots,i_d}^{k,n}$ can be written as a function of the $2p$ values $\left(w_{i_1,\ldots,i_k+j,\ldots,i_D}^n\right)_{j=-p+1}^p$. Furthermore, we assume that the numerical flux function is consistent with $f$ and locally Lipschitz continuous, i.e., for every bounded set $K\subset\R$, there exists a constant $C>0$ such that for $k=1,\ldots,d$,
\begin{equation*}
    \abs{F_{i_1,\ldots,i_k+\hf,\ldots,i_d}^{k,n} - f_k\left(w_{i_1,\ldots,i_d}^n\right)} \leq C \sum_{j=-p+1}^p \abs{w_{i_1,\ldots,i_k+j,\ldots,i_d}^n - w_{i_1,\ldots,i_d}^n}
\end{equation*}
whenever $w_{i_1,\ldots,i_k-p+1,\ldots,i_d}^n,\ldots, w_{i+1,\ldots,i_k+p,\ldots,i_d}^n\in K$. Finally, we consider monotone finite volume methods where the right-hand side of~\eqref{eqn: Finite volume method} is nondecreasing in each argument.

We define the numerical solution operator
\begin{equation*}
    S^\Dx_t\colon \left(\mathrm{L}^1\cap\mathrm{L}^\infty\cap\mathrm{BV}\right)(\R^d)\times\mathcal{C}^{0,1}(\R;\R^d)\to\left(\mathrm{L}^1\cap\mathrm{L}^\infty\cap\mathrm{BV}\right)(\R^d)
\end{equation*}
by
\begin{equation*}
    (S^\Dx_t (\bar{w},f))(x) = w_{i_1,\ldots,i_d}^n, \qquad (x,t)\in \cell_{i_1,\ldots,i_d}\times [t^n,t^{n+1}).
\end{equation*}

The following convergence rate estimate is due to Kutznetsov. 
\begin{theorem}[{\cite[Thm.~4]{kuznetsov76}}]
    Let $\bar{w}\in\left(\mathrm{L}^1\cap\mathrm{L}^\infty\cap\mathrm{BV}\right)(\R^d)$, $f\in\mathcal{C}^{0,1}(\R;\R^d)$, $S_t(\bar{w},f)$ the corresponding entropy solution of~\eqref{eqn: conservation law} and $S^\Dx_t(\bar{w},f)$ the numerical approximation given by~\eqref{eqn: Finite volume method}. Then we have the following convergence rate estimate:
    \begin{equation}
        \norm{S_t(\bar{w},f) - S^\Dx_t(\bar{w},f)}_{\mathrm{L}^1(\R^d)} \leq C \left( \mathrm{TV}(\bar{w}) + \norm{f}_{\mathrm{Lip}}\mathrm{TV}(\bar{w}) \right) \Dx^{\hf}
        \label{eqn: convergence rate}
    \end{equation}
    for all $0\leq t\leq T$ where $C$ is independent of $\Dt,\Dx, \bar{w}$, and $f$.  
\end{theorem}
Note that the convergence rate estimate~\eqref{eqn: convergence rate} is optimal in the sense that the exponent $\hf$ cannot be improved without further assumptions on the initial datum \cite{BadwaikRuf2020,Sabac1997} (see~\cite{Ruf2019} for an overview of the literature regarding optimal convergence rates).

\subsubsection{Bayesian inverse problems for scalar conservation laws}
We will now use the above well-posedness and approximation results to show that the abstract framework of \Cref{sec: well-posedness,sec: approximation} can be applied to Bayesian inverse problems for scalar conservation laws where the inputs $u=(\bar{w},f)$ are inferred from measurements of the observables.
To that end, we define $X=\left(\mathrm{L}^1\cap\mathrm{L}^\infty\cap\mathrm{BV}\right)(\R^d)\times V$ equipped with the norm
\begin{equation*}
    \norm{(\bar{w},f)}_X = \norm{\bar{w}}_{\mathrm{L}^1(\R^d)} + \mathrm{TV}(\bar{w}) + \norm{\bar{w}}_{\mathrm{L}^\infty(\R^d)} + \norm{f}_{V}
\end{equation*}
where $V$ is some separable Banach space embedded in $\mathcal{C}^{0,1}(\R;\R^d)$. Specifically, in light of the Sobolev Embedding Theorem we can take $V= \mathrm{W}^{2,p}(\R;\R^d)$ for any $1<p<\infty$, for example $V=\mathrm{H}^2(\R;\R^d)$.
We then consider observation operators of the form $\mathcal{G}\colon X \to\R^m$ given by
\begin{equation}
    (\mathcal{G}(\bar{w},f))_j = \int_0^T\int_{\R^d} \psi_j(x,t) g_j((S_t(\bar{w},f))(x))\diff x\diff t,\qquad j=1,\ldots,m,
    \label{eqn: observation operator}
\end{equation}
for $\psi_j\in\mathrm{L}^1(\R^d\times(0,T))\cap\mathrm{L}^1(0,T;\mathrm{L}^\infty(\R^d))$ and $g_j\in\mathcal{C}^1(\R;\R^d)$ with $\norm{g_j}_{\mathcal{C}^1(\R;\R^d)}<\infty$.

The following lemma shows that the Bayesian inverse problem of determining the initial datum $\bar{w}$ and the flux function $f$ given observations of the form~\eqref{eqn: observation operator} is well-posed.
\begin{lemma}\label{lem: observation operator satisfies assumptions on G}
    The observation operator $\mathcal{G}$ defined by~\eqref{eqn: observation operator} satisfies \Cref{Assumption on G}. Therefore, by \Cref{cor: finite data well-posedness}, the Bayesian inverse problem associated with the observation operator $\mathcal{G}$ is well-posed.
\end{lemma}

\begin{proof}
    It suffices to consider the case $m=1$. Using the $\mathrm{L}^\infty$ bound~\eqref{eqn: L^infty bound of the ES} of the solution operator we find
    \begin{align*}
        \abs{\mathcal{G}(\bar{w},f)} &\leq \int_0^T\int_{\R^d} \abs{\psi(x,t)} \abs{g((S_t(\bar{w},f))(x))} \diff x\diff t\\
        &\leq \int_0^T\int_{\R^d} \abs{\psi(x,t)} \abs{g((S_t(\bar{w},f))(x)) - g(0)} \diff x\diff t + \int_0^T\int_{\R^d} \abs{\psi(x,t)} \abs{g(0)} \diff x\diff t\\
        &\leq \norm{g'}_\infty \norm{\psi}_{\mathrm{L}^1(\R^d\times(0,T))} \max_{0\leq t\leq T}\norm{S_t(\bar{w},f)}_{\mathrm{L}^\infty(\R^d)} + \norm{g}_\infty \norm{\psi}_{\mathrm{L}^1(\R^d\times(0,T))}\\
        &\leq \norm{g'}_\infty \norm{\psi}_{\mathrm{L}^1(\R^d\times(0,T))} \norm{\bar{w}}_{\mathrm{L}^\infty(\R^d)} + \norm{g}_\infty \norm{\psi}_{\mathrm{L}^1(\R^d\times(0,T))}\\
        &\eqqcolon C_1 \norm{\bar{w}}_{\mathrm{L}^\infty(\R^d)} + C_2.
    \end{align*}
    Let now $\eps >0$. Using the estimates $\ln(x)\leq x$ and $\exp(ax)\leq \exp\left(\eps x^2 +\frac{a^2}{\eps}\right)$ we get
    \begin{align*}
        C_1 \norm{\bar{w}}_{\mathrm{L}^\infty(\R^d)} + C_2 &\leq \exp(C_1\norm{\bar{w}}_{\mathrm{L}^\infty(\R^d)}) + C_2\\
        &\leq (1+ C_2) \exp(C_1\norm{\bar{w}}_{\mathrm{L}^\infty(\R^d)})\\
        &\leq (1+ C_2) \exp\left(\eps\norm{\bar{w}}_{\mathrm{L}^\infty(\R^d)}^2 + \frac{C_1^2}{\eps}\right)\\
        &\leq \exp\left(\eps\norm{(\bar{w},f)}_X^2 + \frac{C_1^2}{\eps} + 1 + C_2\right)
    \end{align*}
    which shows that \Cref{Assumption on G}~\eqref{item: Assumption on G - exponential bound} is satisfied. On the other hand, for $r>0$ and $(\bar{w},f),(\hat{w},g)\in X$ with $\norm{(\bar{w},f)}_X,\norm{(\hat{w},g)}_X<r$, because of~\eqref{eqn: Lipschitz continuity of the solution operator (in f and u_0)} we have
    \begin{align*}
        &\abs{\mathcal{G}(\bar{w},f) - \mathcal{G}(\hat{w},g)}\\
        &\leq \int_0^T\int_{\R^d} \abs{\psi(x,t)} \abs{g((S_t(\bar{w},f))(x)) - g((S_t(\hat{w},g))(x))}\diff x\diff t\\
        &\leq \norm{g'}_\infty \int_0^T \norm{\psi(\cdot,t)}_{\mathrm{L}^\infty(\R^d)} \norm{S_t(\bar{w},f) - S_t(\hat{w},g)}_{\mathrm{L}^1(\R^d)}\diff t\\
        &\leq \norm{\psi}_{\mathrm{L}^1(0,T;\mathrm{L}^\infty(\R^d))} \norm{g'}_\infty \left(\norm{\bar{w}-\hat{w}}_{\mathrm{L}^1(\R^d)} + T\min(\mathrm{TV}(\bar{w}),\mathrm{TV}(\hat{w}))\norm{f-g}_{\mathrm{Lip}}\right)\\
        &\leq \norm{\psi}_{\mathrm{L}^1(0,T;\mathrm{L}^\infty(\R^d))} \norm{g'}_\infty \max(1,CTr) \norm{(\bar{w},f)-(\hat{w},g)}_X
    \end{align*}
    such that \Cref{Assumption on G}~\eqref{item: Assumption on G - Lipshitz continuity in u} is satisfied.
\end{proof}

Using the finite volume method~\eqref{eqn: Finite volume method}, we can define an approximation to $\mathcal{G}$ by replacing the solution operator $S$ in~\eqref{eqn: observation operator} by the numerical solution operator $S^\Dx$,
\begin{equation}
    \big(\mathcal{G}^\Dx(\bar{w},f)\big)_j = \int_0^T\int_{\R^d} \psi_j(x,t) g_j\big((S^\Dx_t(\bar{w},f))(x)\big)\diff x\diff t,\qquad j=1,\ldots,m.
    \label{eqn: approximation of observation operator}
\end{equation}
\begin{lemma}\label{lem: approximation satisfies assumption on approximation}
    The approximation $\mathcal{G}^\Dx$ defined in~\eqref{eqn: approximation of observation operator} of the observation operator $\mathcal{G}$ defined in~\eqref{eqn: observation operator} satisfies~\eqref{eqn: bound for the approximation error in G} in \Cref{cor: finite data approximation} with $\Psi(\Dx^{-1})=\sqrt{\Dx}$.
\end{lemma}
\begin{proof}
    Let $\eps>0$. Using the convergence rate estimate~\eqref{eqn: convergence rate} as well as the estimates $\ln(x)\leq x$ and $\exp(ax)\leq \exp\left(\eps x^2 +\frac{a^2}{\eps}\right)$ we find
    \begin{align*}
        \abs{\mathcal{G}(\bar{w},f)-\mathcal{G}^\Dx(\bar{w},f)} &\leq \norm{g'}_\infty\int_0^T \norm{\psi(\cdot,t)}_{\mathrm{L}^\infty(\R^d)} \norm{S_t(\bar{w},f) - S^\Dx_t(\bar{w},f)}_{\mathrm{L}^1(\R^d)} \diff t\\
        &\leq \norm{g'}_\infty \norm{\psi}_{\mathrm{L}^1(0,T;\mathrm{L}^\infty(\R^d))} C \left(\mathrm{TV}(\bar{w}) + \norm{f}_{\mathrm{Lip}}\mathrm{TV}(\bar{w})\right) \sqrt{\Dx}\\
        &\leq C \left(\norm{(\bar{w},f)}_X + \norm{(\bar{w},f)}_X^2\right) \sqrt{\Dx}\\
        &\leq C \exp\left(\ln(\norm{(\bar{w},f)}_X) + \ln(\norm{(\bar{w},f)}_X +1)\right)\sqrt{\Dx}\\
        &\leq C \exp\left(2\norm{(\bar{w},f)}_X +1\right)\sqrt{\Dx}\\
        &\leq C \exp\left(\eps\norm{(\bar{w},f)}_X^2 + \frac{4}{\eps} +1\right)\sqrt{\Dx}\\
        &= C \exp\left( \frac{4}{\eps} +1\right) \exp\left(\eps\norm{(\bar{w},f)}_X^2\right)\sqrt{\Dx}.
    \end{align*}
\end{proof}

\begin{remark}
    Note that the assertion of \Cref{lem: approximation satisfies assumption on approximation} also holds for $\psi_j=\delta_T$, $j=1,\ldots,m$, where $\delta_T$ is the Dirac delta function. In that case $\mathcal{G}^\Dx$ takes the form
    \begin{equation*}
        \big(\mathcal{G}^\Dx(\bar{w},f)\big)_j = \int_{\R^d} g_j\big((S^\Dx_T(\bar{w},f))(x)\big)\diff x,\qquad j=1,\ldots,m,
    \end{equation*}
    which we will use in \Cref{sec: numerical experiments} for our numerical experiments.
\end{remark}

\subsection{Scalar conservation laws with discontinuous flux in one dimension}
As a second application, we consider the Cauchy problem for scalar conservation laws with discontinuous flux of the form
\begin{gather}
    \begin{aligned}
        w_t + f(k(x),w)_x =0,& &&(x,t)\in \R\times(0,T),\\
        w(x,0) = \bar{w}(x),& &&x\in \R
    \end{aligned}
    \label{eqn: conservation law discontinuous flux}
\end{gather}
where the flux is strictly increasing in $w$ and has a possibly discontinuous spatial dependency through the coefficient $k$.

Note that if the spatial dependency coefficient $k$ is piecewise constant with finitely many discontinuities we effectively consider standard conservation laws where the flux function changes across finitely many points in space. In particular, this includes the important so-called two-flux case
\begin{equation*}
    w_t + \left( H(x)f(w) + (1-H(x))g(w) \right)_x =0
\end{equation*}
where $H$ is the Heaviside function.

\subsubsection{Adapted entropy solutions}
We assume that the flux is strictly increasing in $w$ and consider solutions in the sense of adapted entropy solutions (see \cite{BAITI1997161,audusse2005uniqueness}). To that end, we define for $p\in\R$ the function $c_p\colon\R\to\R$ through the equation
\begin{equation*}
    f(k(x),c_p(x))= p\qquad\text{for all }x\in\R.
\end{equation*}
This equation has a unique solution for each $x\in\R$ since the flux is strictly increasing in $w$.
\begin{definition}[{\cite{BAITI1997161,audusse2005uniqueness}}]\label{def: adapted entropy solution}
    We call a function $w\in\mathrm{L}^\infty(\R\times(0,T))\cap\mathcal{C}([0,T];\mathrm{L}^1(\R))$ an adapted entropy solution of~\eqref{eqn: conservation law discontinuous flux} if for all $p\in\R$
    \begin{multline*}
        \int_0^T\int_\R \left( \abs{w-c_p(x)}\varphi_t + \sign(w-c_p(x))(f(k(x),w) - f(k(x),c_p(x)))\varphi_x \right)\diff x\diff t \\
        + \int_\R \abs{\bar{w}(x)-c_p(x)}\varphi(x,0)\diff x \geq 0
    \end{multline*}
    for all nonnegative $\varphi\in\mathcal{C}_c^\infty(\R\times[0,T))$.
\end{definition}
Since stability results for~\eqref{eqn: conservation law discontinuous flux} with respect to the modeling parameters $\bar{w}$, $k$, and $f$ are only available under the assumption that $k$ is piecewise constant with finitely many discontinuities, we will restrict the exposition to that case from this point on. However, we want to remark that more general results regarding existence and uniqueness of adapted entropy solutions are available in the literature and we refer the reader to \cite{TOWERS20205754,Piccoli/Tournus,audusse2005uniqueness}.
\begin{theorem}\label{thm: existence of adopted entropy solutions}
    %
    %
    Let $f\in\mathcal{C}^2(\R^2;\R)$ be strictly increasing in $w$ in the sense that $f_w\geq \alpha >0$, and assume that $f(k^*,0)=0$ for all $k^*\in\R$. Let further $k$ be piecewise constant with finitely many discontinuities and $\bar{w}\in(\mathrm{L}^\infty\cap\mathrm{BV})(\R)$.
    Then there exists a unique entropy solution $w$ of~\eqref{eqn: conservation law discontinuous flux} and the solution operator $S_t$ given by
    \begin{equation*}
        S_t \bar{w} = w(\cdot,t)
    \end{equation*}
    satisfies
    \begin{romannum}
        \item For all $0\leq t\leq T$
            \begin{align*}
                \norm{S_t \bar{w}}_{\mathrm{L}^1(\R)} &\leq \norm{\bar{w}}_{\mathrm{L}^1(\R)},\\
                \norm{S_t \bar{w}}_{\mathrm{L}^\infty(\R)} &\leq \frac{C_f}{\alpha}\norm{\bar{w}}_{\mathrm{L}^\infty(\R)},
                \shortintertext{and}
                \mathrm{TV} (S_t \bar{w}) &\leq C(\mathrm{TV}(\bar{w}) + \mathrm{TV}(k))
            \end{align*}
            where $C_f$ denotes the maximal Lipschitz constant of $f$.
        \item For all $x\in\R$
            \begin{equation*}
                \mathrm{TV}_{[0,T]}(t\mapsto (S_t \bar{w})(x)) \leq C\mathrm{TV}(\bar{w}).
            \end{equation*}
    \end{romannum}
\end{theorem}
\begin{proof}
    The existence and uniqueness statement follows from the theory developed by Baiti and Jenssen~\cite{BAITI1997161}. The $\mathrm{L}^1$, $\mathrm{L}^\infty$, and $\mathrm{TV}$ bounds follow from~\cite[Thm.~4.1]{ruf2021fluxstability},~\cite[Thm.~1.4]{TOWERS20205754}, and~\cite[Lem.~4.6]{BadwaikRuf2020} respectively.
\end{proof}

Similarly to before, we will denote the solution operator by $S_t (\bar{w},k,f)$ to highlight the dependence on $k$ and $f$ as well. We have the following Lipschitz continuity result.

\begin{theorem}[{\cite[Thm.~4.1]{ruf2021fluxstability}}]
    Let $f$ and $g$ be flux functions satisfying the assumptions of \Cref{thm: existence of adopted entropy solutions}, $k$ and $l$ be piecewise constant functions with finitely many discontinuities and $\bar{w},\hat{w}\in(\mathrm{L}^\infty\cap\mathrm{BV})(\R)$. Then the solution operator satisfies
    \begin{equation}
        \norm{S_t(\bar{w},k,f) - S_t(\hat{w},l,g)}_{\mathrm{L}^1(\R)} \leq \norm{\bar{w}-\hat{w}}_{\mathrm{L}^1(\R)} + C\left(\norm{k-l}_{\mathrm{L}^\infty(\R)} + \norm{f_w-g_w}_{\mathrm{L}^\infty(\R^2;\R)} \right).
        \label{eqn: stability estimate discontinuous flux}
    \end{equation}
    for every $0\leq t\leq T$.
\end{theorem}
Note that the constant $C$ in~\eqref{eqn: stability estimate discontinuous flux} depends linearly on (products of) the $\mathrm{L}^\infty$ and $\mathrm{TV}$ norms of $\bar{w}$ and $\hat{w}$, the Lipschitz constants of $f$ and $g$ and the maximum number of discontinuities in $k$ and $l$.

\subsubsection{Finite volume methods}
We will now present a class of finite volume methods for~\eqref{eqn: conservation law discontinuous flux} introduced in~\cite{BadwaikRuf2020}.
As before, we discretize the domain $\R\times[0,T]$ using the spatial and temporal grid discretization parameters $\Dx$ and $\Dt$. The resulting grid cells we denote by $\cell_j=(x_{j+\hf},x_{j-\hf})$ in space and $[t^n,t^{n+1})$ in time for points $x_{j+\hf}$, such that $x_{j+\hf}-x_{j-\hf}=\Dx$, $j\in\Z$, and $t^n = n\Dt$ for $n=0,\ldots,M+1$.

For a given coefficient $k$ we denote by $\xi_i$, $i=1,\ldots,N$, its discontinuities and by $D_i=(\xi_i,\xi_{i+1})$, $i=0,\ldots,N$, the subdomains where $k$ is constant. Here we have used the notation $\xi_0=-\infty$ and $\xi_{N+1}=+\infty$. Furthermore, we will write
\begin{equation*}
    f^{(i)} = f(k(x),\cdot),\qquad \text{for }x\in D_i,\ i=0,\ldots,N.
\end{equation*}

In the following, we will assume that the grid is aligned in such a way that all discontinuities of $k$ lie on cell interfaces, i.e., $\xi_i=x_{P_i -\hf}$ for some integers $P_i$, $i=1,\ldots,N$. In general, this can be achieved by considering a globally nonuniform grid that is uniform on each $D_i$ and taking $\Dx = \max_{i=0,\ldots,N} \Dx_i$ where $\Dx_i$ is the grid discretization parameter in $D_i$.

The finite volume method we consider is the following \cite{BadwaikRuf2020}:
\begin{gather}
  \begin{aligned}
    w_j^{n+1} = w_j^n - \lambda\left( f^{(i)}(w_j^n) - f^{(i)}(w_{j-1}^n) \right),& &&n\geq 0,\ P_i < j < P_{i+1},\ 0\leq i\leq N,\\
    w_{P_i}^{n+1} = \left(f^{(i)}\right)^{-1}\left(f^{(i-1)}\left(w_{P_i -1}^{n+1}\right)\right),& &&n\geq 0,\ 0<i\leq N,\\
    w_j^0 = \frac{1}{\Dx}\int_{\cell_j} \bar{w}(x)\diff x,& &&j\in\Z,
  \end{aligned}
  \label{eqn: FVM discontinuous flux}
\end{gather}
where $P_0=-\infty$, $P_{N+1}=+\infty$, and $\lambda = \Dt/\Dx$. We assume that the grid discretization parameters satisfy the following CFL condition:
\begin{equation}
  \lambda\max_i \max_u \left(f^{(i)}\right)'(w) \leq 1.
  \label{eqn: CFL condition}
\end{equation}
Note that the definition of $w_{P_i}^{n+1}$ in~\eqref{eqn: FVM discontinuous flux} represents a discrete version of the Rankine--Hugoniot condition which in the setting of conservation laws with discontinuous flux holds across discontinuities of $k$. Here, we use the ghost cells $\cell_{P_i}$, $i=1,\ldots,N$ to explicitly enforce the Rankine--Hugoniot condition on the discrete level.

We define the numerical solution operator $S_t^\Dx$ by
\begin{equation*}
    (S_t^\Dx (\bar{w},k,f))(x) = w_j^n,\qquad (x,t)\in\cell_j\times[t^n,t^{n+1}).
\end{equation*}
The following lemma shows that the finite volume method is stable in $\mathrm{L}^\infty$ and $\mathrm{L}^1$.
\begin{lemma}[{\cite[Lem.~5.1]{badwaik2020multilevel}}]
  Let $f,k$, and $\bar{w}$ satisfy the assumptions of \Cref{thm: existence of adopted entropy solutions}.
  If the numerical scheme~\eqref{eqn: FVM discontinuous flux} satisfies the CFL condition~\eqref{eqn: CFL condition} we have the following stability estimates:
  \begin{align*}
    \norm{S_t^\Dx(\bar{w},k,f)}_{\mathrm{L}^\infty(\R)} &\leq \frac{C_f}{\alpha}\norm{\bar{w}}_{\mathrm{L}^\infty(\R)}
    \shortintertext{and}
    \norm{S_t^\Dx(\bar{w},k,f)}_{\mathrm{L}^1(\R)} &\leq \norm{\bar{w}}_{\mathrm{L}^1(\R)} + C\mathrm{TV}(\bar{w})\Dx.
  \end{align*}
\end{lemma}

\begin{theorem}[{\cite[Thm.~5.1]{BadwaikRuf2020}}]
    Let $f,k$, and $\bar{w}$ satisfy the assumptions of \Cref{thm: existence of adopted entropy solutions}.
    Let $S_t(\bar{w},k,f)$ denote the corresponding adapted entropy solution of~\eqref{eqn: conservation law discontinuous flux} and $S_t^\Dx(\bar{w},k,f)$ the numerical approximation given by~\eqref{eqn: FVM discontinuous flux}. Then we have the following convergence rate estimate
    \begin{equation}
        \norm{S_t (\bar{w},k,f) - S_t^\Dx(\bar{w},k,f)}_{\mathrm{L}^1(\R)} \leq C \Dx^{\hf}
    \label{eqn: Convergence rate discontinuous flux}
    \end{equation}
    for all $0\leq t\leq T$. Like in~\eqref{eqn: convergence rate}, the constant $C$ depends polynomially on $\mathrm{TV}(\bar{w}), \norm{f}_{\mathrm{Lip}}$ and in this case the number of discontinuities of $k$.
\end{theorem}

\subsubsection{Bayesian inverse problems for scalar conservation laws with discontinuous flux}

We consider a given, fixed set of points
    $\xi_1 < \xi_2 < \ldots <\xi_{N-1}$
for $N\gg 1$ representing the possible points of discontinuity of the coefficient $k$.
We identify the space
\begin{equation*}
    V \coloneqq \{k\in\mathrm{L}^\infty(\R)\mid k \text{ is piecewise constant with discontinuities among the points }\xi_1,\ldots,\xi_{N-1} \}
 \end{equation*}
 (as a subspace of $\mathrm{L}^\infty(\R)$)
 with $(\R^N,\norm{\cdot}_\infty)$ by associating $k\in V$ with the vector $(k_i)_{i=1}^N\in\R^N$ representing the values of $k$ between neighboring points $\xi_i$ and $\xi_{i+1}$. We then consider the Bayesian inverse problem with $X = (\mathrm{L}^\infty\cap\mathrm{BV})(\R) \times \R^N$ and define the observation operator $\mathcal{G}\colon X\to\R^m$ by
 \begin{equation}
    (\mathcal{G}(\bar{w},k))_j = \int_0^T\int_\R \psi_j(x,t) g_j((S_t(\bar{w},k))(x))\diff x\diff t,\qquad j=1,\ldots,m,
    \label{eqn: observation operator discontinuous flux}
 \end{equation}
 for $\psi_j\in\mathrm{L}^1(\R^d\times(0,T))\cap\mathrm{L}^1(0,T;\mathrm{L}^\infty(\R^d))$ and $g_j\in\mathcal{C}^1(\R;\R^d)$ with $\norm{g_j}_{\mathcal{C}^1(\R;\R^d)}<\infty$. Note that here we keep the flux $f$ fixed since the assumption $f_w\geq \alpha>0$ is incompatible with a Banach space setting.
As before, we use the finite volume method~\eqref{eqn: FVM discontinuous flux} to define an approximation to $\mathcal{G}$ in the following way:
\begin{equation}
    \big(\mathcal{G}^\Dx(\bar{w},k)\big)_j = \int_0^T\int_\R\psi_j(x,t) g_j\big((S_t^\Dx(\bar{w},k))(x)\big)\diff x\diff t,\qquad j=1,\ldots,m.
    \label{eqn: approximation of observation operator discontinuous flux}
\end{equation}

\begin{lemma}
    The observation operator defined by~\eqref{eqn: observation operator discontinuous flux} satisfies \Cref{Assumption on G} and the approximation $\mathcal{G}^\Dx$ satisfies~\eqref{eqn: bound for the approximation error in G} with $\Psi(\Dx^{-1})=\sqrt{\Dx}$.
\end{lemma}
\begin{proof}
    In light of the stability estimate~\eqref{eqn: stability estimate discontinuous flux} and the convergence rate~\eqref{eqn: Convergence rate discontinuous flux}, the proof can be carried out in the same way, \emph{mutatis mutandis}, as the proofs of \Cref{lem: observation operator satisfies assumptions on G,lem: approximation satisfies assumption on approximation}. 
\end{proof}




\section{Numerical experiments}\label{sec: numerical experiments}

In this section, we illustrate our theoretical results by presenting a series of numerical experiments. We employ a Metropolis--Hastings method to generate a Markov chain which samples from the posterior $\mu^y$. Such methods require a proposal kernel and here we choose the following standard random walk (see \cite{stuart2010inverse,cotter2013mcmc}):
\begin{itemize}
    \item Set $n=0$ and pick $u^{(0)}$.
    \item Propose $v^{(n)}=u^{(n)} + \beta \xi^{(n)}$ where $\xi^{(n)}\sim\mathcal{N}(0,\mathcal{C})$.
    \item Set $u^{(n+1)} = v^{(n)}$ with probability $a(u^{(n)},v^{(n)})$.
    \item Set $u^{(n+1)}=u^{(n)}$ otherwise.
    \item $n\to n+1$.
\end{itemize}
The underlying acceptance probability is defined as
\begin{equation*}
    a(u,v) = \min ( 1, \exp(I(u)-I(v)))
\end{equation*}
where
\begin{equation*}
    I(u) = \Phi(u) + \frac{1}{2}\norm{\mathcal{C}^{-\unitfrac{1}{2}}u}_X^2.
\end{equation*}
If we generate $\xi^{(n)}$ and the uniform random variable used in the accept-reject step independently of each other for each $n$ and independently of their values for different $n$ then this construction gives rise to a Markov chain $(u^{(n)})_{n=0}^\infty$ which is distributed according to $\mu^y$ given by \eqref{eqn: Radon--Nikodym derivative} \cite{stuart2010inverse}.

The algorithm has three scalar hyperparameters which need to be specified. First, the stepsize $\beta$ which controls the size of the move, second the burn-in $b$, i.e., the number of samples which are discarded in order to minimize the contribution of the initial value $u^{(0)}$, and the sample interval $\tau$ which is the number of states which are discarded between two observations.



 The best choices of hyperparameters, corresponding to short burn-in and smaller step-size in the steady-state can be achieved by letting $\beta$ vary with the step-count, i.e., $\beta = \beta (k)$. We chose a piecewise linear function for $\beta$, where in the beginning the steps are large and decrease linearly until a certain number of steps, after which it stays constant. Ultimately, the step-size is problem-dependent and has to be adjusted for each problem, for instance by a grid search. 
 

\subsection{Inverse problem for the shock location and amplitude in a Riemann problem for Burgers' equation}
In our first numerical experiment we consider Burgers' equation
\begin{gather*}
\begin{aligned}
    w_t + \Big(\frac{w^2}{2}\Big)_x = 0,& &&(x,t)\in(-1,1)\times(0,T),\\
    w(x,0)=\bar{w}(x),& &&x\in(-1,1),
\end{aligned}
\end{gather*}
with outflow boundary conditions.
Given numerical solutions at a specified time we want to infer the initial datum which we assume is of the form
\begin{equation*}
    \bar{w}^{(\delta_1,\delta_2,\sigma_0)}(x) = \begin{cases}
        1 + \delta_1, &  x<\sigma_0,\\
        \delta_2, & x>\sigma_0,
    \end{cases}
\end{equation*}
parameterized by $u=(\delta_1,\delta_2,\sigma_0)\in\R^3$. In order to infer the parameters $(\delta_1,\delta_2,\sigma_0)$ by observing (an approximation of) the solution $w$ at time $T=1$ we define the observation operator
\begin{equation*}
    \left(\mathcal{G}(\delta_1,\delta_2,\sigma_0))\right)_j = 10\int_{x_j - 0.05}^{x_j + 0.05} S_T^\Dx\left(\bar{w}^{(\delta_1,\delta_2,\sigma_0)}\right)\diff x, \qquad j=1,\ldots,5,
\end{equation*}
where $S_T^\Dx$ denotes the numerical solution operator and the measurement points are $(x_j)_{j=1}^5 = (-0.5,-0.25,0.35,0.5,0.65)$.
\begin{figure}[t]
\centering
\begin{tikzpicture}
  \begin{axis}[xscale=1.5,ymin=-0.2,ymax=1.2,xtick={-1,-0.5,0,0.5,1},xticklabels={$-1$,$-0.5$,$0$,$0.5$,$1$},ytick={0.,1},yticklabels={$0$,$1$}]
    \fill[green!60!black, opacity=0.5] (-0.45,-10) -- (-0.55,-10) -- (-0.55,10) -- (-0.45,10) -- (-0.45,-10);
    \fill[green!60!black, opacity=0.5] (-0.2,-10) -- (-0.3,-10) -- (-0.3,10) -- (-0.2,10) -- (-0.2,-10);
    \fill[green!60!black, opacity=0.5] (0.3,-10) -- (0.4,-10) -- (0.4,10) -- (0.3,10) -- (0.3,-10);
    \fill[green!60!black, opacity=0.5] (0.45,-10) -- (0.55,-10) -- (0.55,10) -- (0.45,10) -- (0.45,-10);
    \fill[green!60!black, opacity=0.5] (0.6,-10) -- (0.7,-10) -- (0.7,10) -- (0.6,10) -- (0.6,-10);
    \addplot[myblue, dotted, very thick,mark=none, const plot] table {Experiments/Experiment1_Burgers_RP/Burgers_prior_mean_ID.txt};
    \addplot[myblue, very thick,mark=none] table {Experiments/Experiment1_Burgers_RP/Burgers_prior_mean_h=128.txt};
    \addplot[scarletred1, dotted, very thick, mark=none] table {Experiments/Experiment1_Burgers_RP/Burgers_ground_truth_ID.txt};
    \addplot[scarletred1, very thick,mark=none] table {Experiments/Experiment1_Burgers_RP/Burgers_ground_truth_h=128.txt};
  \end{axis}
\end{tikzpicture}
\caption{\emph{Experiment 1:} Initial data (dotted lines) and numerical solutions (solid lines) for $(\delta_1^p,\delta_2^p,\sigma_0^p)=(0.1,-0.1,-0.1)$ (blue) corresponding to the prior mean and $(\delta_1^*,\delta_2^*,\sigma_0^*)=(0,0,0)$ (red) corresponding to the ground truth. The numerical solutions are calculated using the Rusanov scheme and the grid discretization parameter $\Dx=2/128$.}
\label{fig: Burgers experiment setup}
\end{figure}
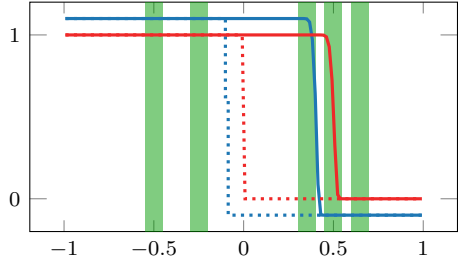

We consider observational noise $\eta \sim \mathcal{N}(0,\gamma^2 \mathcal{I}_5)$ with $\gamma = 0.05$ and prior $\mu_0 \sim \mathcal{N}(u^p,\mathcal{C})$ with mean $u^p=(\delta_1^p,\delta_2^p,\sigma_0^p)=(0.1,-0.1,-0.1)$ and covariance matrix $\mathcal{C} = \varphi^2 \mathcal{I}_3$, $\varphi=0.15$. The ground truth we want to recover is $u^* = (\delta_1^*,\delta_2^*,\sigma_0^*)=(0,0,0)$. \Cref{fig: Burgers experiment setup} shows the initial data corresponding to the prior mean and the ground truth as well as corresponding numerical solutions computed at time $T=1$. The measurement intervals used in the observation operator are highlighted in green.

As for the step size used in the Metropolis--Hastings method, we chose
\begin{equation}
    \beta(k) = \begin{cases}
        \beta_0 - \frac{\beta_0 - \beta_1}{k_b} k , & k \leq k_b, \\
        \beta_1, & k > k_b,
    \end{cases}
    \label{eq: beta}
\end{equation}
where $(\beta_0, \beta_1, k_b) = (0.05, 0.001, 250)$. This allowed
us to use $b = 500$ and $\tau = 20$, i.e., after discarding the first $500$ states use every $20$th state to approximate the posterior.
\begin{figure}[t]
\centering
\subfloat[$\delta_1$]{
\begin{tikzpicture}
    \begin{axis}[xmin=-0.3,xmax=0.3,ymin=0,ymax=15,cycle list name=mycycle,xtick={-0.25,0,0.25},xticklabels={$-0.25$,$0$,$0.25$}]
        \addplot+[ybar, bar width=0.0139,fill, opacity=0.6, draw=none] table{Experiments/Experiment1_Burgers_RP/hist_delta_1_h=128.tikz.txt};
        \addplot+[very thick, domain=-0.6:0.6, samples=100]{1/(0.15*sqrt{(2*pi)})*exp(-0.5*((x-0.1)/0.15)^2)};
        \addplot+[ybar, bar width=.4pt, fill] coordinates{
        (0, 100)
        };
    \end{axis}
\end{tikzpicture}
}
\subfloat[$\delta_2$]{
\begin{tikzpicture}
    \begin{axis}[xmin=-0.3,xmax=0.3,ymin=0,ymax=15,cycle list name=mycycle,xtick={-0.25,0,0.25},xticklabels={$-0.25$,$0$,$0.25$}]
        \addplot+[ybar, bar width=0.0139,fill, opacity=0.6, draw=none] table{Experiments/Experiment1_Burgers_RP/hist_delta_2_h=128.tikz.txt};
        \addplot+[very thick, domain=-0.6:0.6, samples=100]{1/(0.15*sqrt{(2*pi)})*exp(-0.5*((x+0.1)/0.15)^2)};
        \addplot+[ybar, bar width=.4pt, fill] coordinates{
        (0, 100)
        };
    \end{axis}
\end{tikzpicture}
}
\subfloat[$\sigma_0$]{
\begin{tikzpicture}
    \begin{axis}[xmin=-0.3,xmax=0.3,ymin=0,ymax=15,cycle list name=mycycle,xtick={-0.25,0,0.25},xticklabels={$-0.25$,$0$,$0.25$}]
        \addplot+[ybar, bar width=0.0139,fill, opacity=0.6, draw=none] table{Experiments/Experiment1_Burgers_RP/hist_sigma_0_h=128.tikz.txt};
        \addplot+[very thick, domain=-0.6:0.6, samples=100]{1/(0.15*sqrt{(2*pi)})*exp(-0.5*((x+0.1)/0.15)^2)};
        \addplot+[ybar, bar width=.4pt, fill] coordinates{
        (0, 100)
        };
    \end{axis}
\end{tikzpicture}
}
\caption{\emph{Experiment 1:} Histograms corresponding to the Metropolis-Hastings approximation of the posterior using a chain length of $2500$ and $\Dx=2/128$. The ground truth and the prior are shown in red and blue respectively.}
\label{exp1: histograms}
\end{figure}
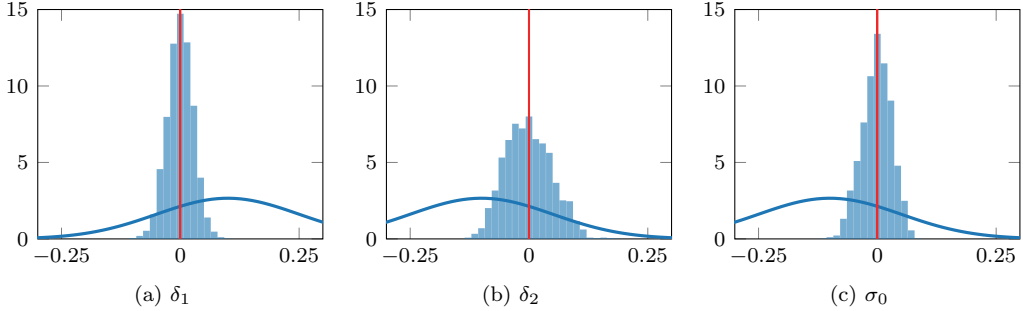

\Cref{exp1: histograms} shows the histograms of the approximated posterior computed by the Metropolis--Hastings method with a chain length of $2500$ and using $\Dx=2/128$ for the underlying finite volume method for the forward problem. The resulting posteriors all peak at the ground truth parameter values. The posteriors indicate the uncertainties inherent in estimating these parameters. The posterior of $\delta_2$ has the largest spread indicating comparatively slightly larger uncertainty in this parameter. This appears to be a consequence of the placement of the measurement intervals since only the rightmost measurement interval around the point $x_5=0.75$ contributes towards inferring the parameter $\delta_2$. The mean of the approximated posterior is $u_{\text{mean}}=(\delta_1,\delta_2,\sigma_0)\approx(-0.0004, -0.0010, -0.0012)$ and the \emph{maximum a posteriori} (MAP) estimator is $u_{\text{MAP}}=(\delta_1,\delta_2,\sigma_0)\approx(0.0136 -0.0195, -0.0037)$ both very close to the ground truth $u^*=(0,0,0)$.
%
%
%
%
%
%
%
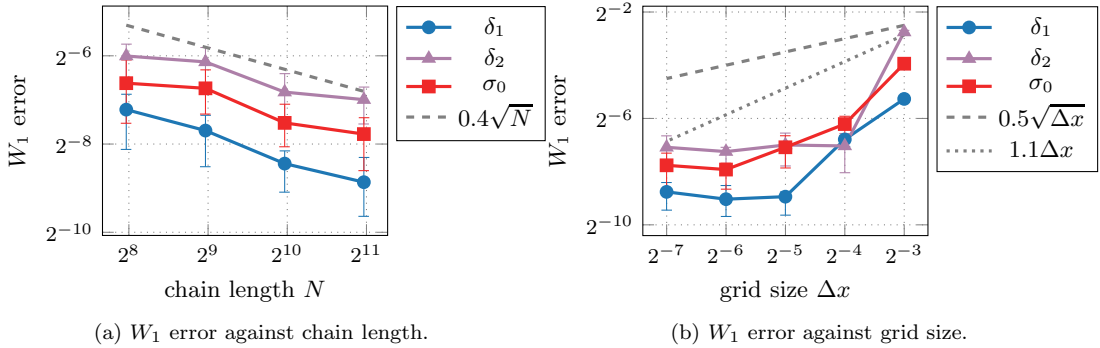
\begin{figure}[t]
\centering
\subfloat[$W_1$ error against chain length.]{
    \begin{tikzpicture}
        \begin{loglogaxis}[
    log basis x={2},
    xlabel={\small chain length $N$},
    log basis y={2},
    ylabel={\small $W_1$ error},
    grid=major,
    grid style={gray, dotted},
    legend pos= outer north east,
    legend style={font=\small},
    cycle list name=mycycle2,
]
            \addplot+[
                very thick,
                error bars/.cd,
                y explicit,
                y dir=both,
            ] table [
                x=steps,
                y=W1,
                y error plus expr=\thisrow{upper_quartile}-\thisrow{W1},
                y error minus expr=\thisrow{W1}-\thisrow{lower_quartile},
            ] {Experiments/Experiment1_Burgers_RP/W_1_for_different_chain_lengths_for_delta_1.tikz.txt};
            \addlegendentry{$\delta_1$}
            \addplot+[
                very thick,
                error bars/.cd,
                    y explicit,
                    y dir=both,
            ] table [
                x=steps,
                y=W1,
                y error plus expr=\thisrow{upper_quartile}-\thisrow{W1},
                y error minus expr=\thisrow{W1}-\thisrow{lower_quartile},
            ] {Experiments/Experiment1_Burgers_RP/W_1_for_different_chain_lengths_for_delta_2.tikz.txt};
            \addlegendentry{$\delta_2$}
        \addplot+[
                very thick,
                error bars/.cd,
                    y explicit,
                    y dir=both,
            ] table [
                x=steps,
                y=W1,
                y error plus expr=\thisrow{upper_quartile}-\thisrow{W1},
                y error minus expr=\thisrow{W1}-\thisrow{lower_quartile},
            ] {Experiments/Experiment1_Burgers_RP/W_1_for_different_chain_lengths_for_sigma_0.tikz.txt};
            \addlegendentry{$\sigma_0$}
            \addplot[very thick,gray, dashed, domain = 250:2000] {0.4/sqrt(x)};
            \addlegendentry{$0.4\sqrt{N}$}
        \end{loglogaxis}
    \end{tikzpicture}
    \hspace{-1.2em}
}
\subfloat[$W_1$ error against grid size.]{
    \begin{tikzpicture}
        \begin{loglogaxis}[
    log basis x={2},
    xlabel={\small grid size $\Dx$},
    log basis y={2},
    ylabel={\small $W_1$ error},
    grid=major,
    grid style={gray, dotted},
    legend pos= outer north east,
    legend style={font=\small},
    cycle list name=mycycle2,
]
            \addplot+[
                very thick,
                error bars/.cd,
                    y explicit,
                    y dir=both,
            ] table [
                x=cells,
                y=W1,
                y error plus expr=\thisrow{upper_quartile}-\thisrow{W1},
                y error minus expr=\thisrow{W1}-\thisrow{lower_quartile},
            ] {Experiments/Experiment1_Burgers_RP/W_1_for_different_grid_spacings_for_delta_1.tikz.txt};
            \addlegendentry{$\delta_1$}
            \addplot+[
                very thick,
                error bars/.cd,
                    y explicit,
                    y dir=both,
            ] table [
                x=cells,
                y=W1,
                y error plus expr=\thisrow{upper_quartile}-\thisrow{W1},
                y error minus expr=\thisrow{W1}-\thisrow{lower_quartile},
            ] {Experiments/Experiment1_Burgers_RP/W_1_for_different_grid_spacings_for_delta_2.tikz.txt};
            \addlegendentry{$\delta_2$}
        \addplot+[
                very thick,
                error bars/.cd,
                    y explicit,
                    y dir=both,
            ] table [
                x=cells,
                y=W1,
                y error plus expr=\thisrow{upper_quartile}-\thisrow{W1},
                y error minus expr=\thisrow{W1}-\thisrow{lower_quartile},
            ] {Experiments/Experiment1_Burgers_RP/W_1_for_different_grid_spacings_for_sigma_0.tikz.txt};
            \addlegendentry{$\sigma_0$}
        \addplot[very thick,gray, dashed, domain = 0.125:0.0078125] {0.5*sqrt(x)};
        \addlegendentry{$0.5\sqrt{\Dx}$}
        \addplot[very thick,gray, dotted, domain = 0.125:0.0078125] {1.1*x};
        \addlegendentry{$1.1\Dx$}
        \end{loglogaxis}
    \end{tikzpicture}
}
\caption{\emph{Experiment 1:} $W_1$ error as a function of the chain length with a fixed grid discretization parameter $\Dx=2/128$  (left) and as a function of the grid discretization parameter $\Dx$ for a fixed chain length ($N=2500$) (right)}
\label{fig: convergence rates}
\end{figure}

In \Cref{fig: convergence rates} we investigate the convergence of the approximated posterior measured in the $1$-Wasserstein distance with respect to the length of the chain as well as with respect to the grid discretization parameter $\Dx$ used in the finite volume method of the forward problem. Specifically, in \Cref{fig: convergence rates} (a) we consider chain lengths $250, 500, 1000,$ and $2000$ while keeping the grid discretization parameter $\Dx=2/128$ constant. On the other hand, in \Cref{fig: convergence rates} (b) we use $16, 32, 64, 128,$ and $256$ cells in the domain $(-1,1)$ while keeping the chain length $N=2500$ constant. We compute each Wasserstein error shown in \Cref{fig: convergence rates} as
\begin{equation*}
    \frac{1}{K}\sum_{k=1}^K W_1\left(U_k^{N,\Dx},U_{\text{Ref}}^{N^*,\Dx^*}\right)
\end{equation*}
where $\big(U_k^{N,\Dx}\big)_{k=1}^K$ is an ensemble of $K$ Markov chains all of length $N$ and using the same grid discretization parameter $\Dx$ and $U_{\text{Ref}}^{N^*,\Dx^*}$ is a reference solution. In the case of convergence with respect to the chain length we computed the reference solution $U_{\text{Ref}}^{N^*,\Dx^*}$ as an average of an ensemble of $K$ Markov chains using $N^*=4000$ and $\Dx^*=2/128$. For the convergence with respect to $\Dx$ we computed $U_{\text{Ref}}^{N^*,\Dx^*}$ again as an average of $K$ Markov chains using $N^*=2500$ and $\Dx^*=2/512$. In both experiments we used an ensembles of size $K=30$.

\Cref{fig: convergence rates} shows that both errors decrease at approximately the expected rate (for the grid size the expected rate is $\sqrt{\Dx}$, cf. \Cref{lem: approximation satisfies assumption on approximation} and \Cref{cor: finite data approximation}). It is clear from Figure \ref{fig: convergence rates} (right) that there is a saturation of convergence with respect to some finer grid sizes. This can be explained by the fact that the sampling error with respect to the chain length (see \Cref{fig: convergence rates} (left)) has already been reached and dominates the discretization error due to the numerical method.

\subsection{Inverse problem for the transport speed and jump amplitude for a Riemann problem with flux discontinuity}
In our second experiment we consider the conservation law with discontinuous flux
\begin{equation}
    w_t + (k(x) f(w) + (1-k(x)) g(w))_x =0
    \label{exp2: discontinuous flux}
\end{equation}
where $k$ is the Heaviside function and $g$ and $f$ are the Transport respectively Burgers flux, i.e.,
    $g^{(a)}(w) = a w,$ and $f(w) = \frac{w^2}{2}$.
Equation \eqref{exp2: discontinuous flux} corresponds to switching from the Transport equation to Burgers equation across the flux interface at $x=0$.
We use the initial datum
\begin{equation*}
    \bar{w}^{(\delta)}(x) = \begin{cases}
        0.5 + \delta,& x<-0.5,\\
        2,& x>-0.5,
    \end{cases}
\end{equation*}
on the domain $(-1,1)$ with outflow boundary conditions and our aim is to infer the left state of the Riemann initial datum, i.e., $\delta$, as well as the transport speed $a$ by observing the (numerical) solution at time $T=1$. Specifically, we consider the observation operator
\begin{equation*}
    \big(\mathcal{G}\big(\delta,a)\big)_j = 10\int_{x_j-0.075}^{x_j+0.075} S_T^\Dx\big(\bar{w}^{(\delta)},g^{(a)}\big) \diff x,\qquad j=1,\ldots,6,
\end{equation*}
where $S_T^\Dx$ is the numerical solution operator defined in~\eqref{eqn: FVM discontinuous flux} and the measurement points are $(x_j)_{j=1}^6= (-0.5,0.1,0.3,0.5,0.7,0.9)$.
We consider observational noise $\eta\sim\mathcal{N}(0,\gamma^2\mathcal{I}_6)$ with $\gamma =0.05$ and prior $\mu_0\sim(u^p,\mathcal{C})$ with mean $u^p=(\delta^p,a^p)=(0.1,0.9)$ and covariance matrix $\mathcal{C}=\varphi^2\mathcal{I}_2$, $\varphi=0.15$. The ground truth we want to recover is $(\delta^*,a^*)=(0,1)$.
\Cref{fig: experiment 3 setup} illustrates the initial data and numerical solutions corresponding to the prior mean and ground truth parameters.

\Cref{exp3: histograms} shows the histograms of the approximated posterior. Here, we used a chain length of $2500$ and $\Dx=128$ and $\lambda = 0.4$ in the finite volume approximation~\eqref{eqn: FVM discontinuous flux}. The mean of the approximated posterior is $u_{\text{mean}}=(\delta,a)\approx(-0.00732, 1.00119)$ and the MAP estimator us $u_{\text{MAP}}=(\delta,a)\approx(-0.00414, 1.00076)$.

\Cref{fig: convergence rates 2} again illustrates the convergence of the approximated posterior measured in the $1$-Wasserstein distance with respect to the length of the chain and with respect to the grid discretization parameter $\Dx$. We see that the observed order of convergence with respect to $\Dx$ in this experiment is strictly higher than the order $\sqrt{\Dx}$ which our theory guarantees. This observation is in line with the fact that the experimental order of convergence for finite volume methods is typically closer to one.
%
%
%
%
%
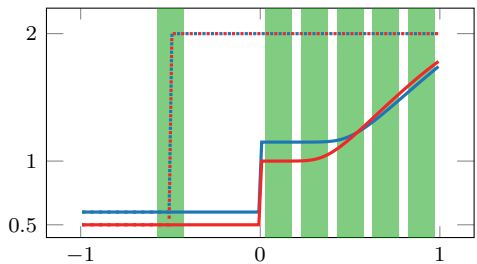
\begin{figure}[t]
\centering
\begin{tikzpicture}
  \begin{axis}[xscale=1.5,ymin=0.4,ymax=2.2,xtick={-1, 0, 1},xticklabels={$-1$,$0$,$1$},ytick={0.5,1,2},yticklabels={$0.5$,$1$,$2$}]
    \fill[green!60!black, opacity=0.5] (-0.575,-10) -- (-0.425,-10) -- (-0.425,10) -- (-0.575,10) -- (-0.575,-10);
    \fill[green!60!black, opacity=0.5] (0.025,-10) -- (0.175,-10) -- (0.175,10) -- (0.025,10) -- (0.025,-10);
    \fill[green!60!black, opacity=0.5] (0.225,-10) -- (0.375,-10) -- (0.375,10) -- (0.225,10) -- (0.225,-10);
    \fill[green!60!black, opacity=0.5] (0.425,-10) -- (0.575,-10) -- (0.575,10) -- (0.425,10) -- (0.425,-10);
    \fill[green!60!black, opacity=0.5] (0.625,-10) -- (0.775,-10) -- (0.775,10) -- (0.625,10) -- (0.625,-10);
    \fill[green!60!black, opacity=0.5] (0.825,-10) -- (0.975,-10) -- (0.975,10) -- (0.825,10) -- (0.825,-10);
    \addplot[myblue, dotted, very thick,mark=none] table {Experiments/Experiment3_Discontinuous_flux/Discontinuous_flux_prior_mean_ID.txt};
    \addplot[myblue, very thick,mark=none] table {Experiments/Experiment3_Discontinuous_flux/Discontinuous_flux_prior_mean.txt};
    \addplot[scarletred1, dotted, very thick, mark=none] table {Experiments/Experiment3_Discontinuous_flux/Discontinuous_flux_ID.txt};
    \addplot[scarletred1, very thick,mark=none] table {Experiments/Experiment3_Discontinuous_flux/Discontinuous_flux_ground_truth.txt};
  \end{axis}
\end{tikzpicture}
\caption{\emph{Experiment 2:} Initial data (dotted lines) and numerical solutions (solid lines) for $\delta=0.1$ and $a=0.9$ (blue) corresponding to the prior mean and $\delta=0$ and $a=1.$ (red) corresponding to the ground truth. The numerical solutions are calculated using the scheme~\eqref{eqn: FVM discontinuous flux} with $\lambda=0.4$ and grid discretization parameter $\Dx=2/128$.}
\label{fig: experiment 3 setup}
\end{figure}
\begin{figure}[t]
\centering
\subfloat[$\delta$]{
\begin{tikzpicture}
    \begin{axis}[xmin=-0.3,xmax=0.3,ymin=0,ymax=20,cycle list name=mycycle,xtick={-0.25,0,0.25},xticklabels={$-0.25$,$0$,$0.25$}]
        \addplot+[ybar, bar width=0.0139, fill,opacity=0.6, draw=none] table{Experiments/Experiment3_Discontinuous_flux/hist_discontinuousflux_delta_h=128.tikz.txt};
        \addplot+[very thick, domain=-0.6:0.6, samples=100]{1/(0.15*sqrt{(2*pi)})*exp(-0.5*((x-0.1)/0.15)^2)};
        \addplot+[ybar, bar width=.4pt, fill] coordinates{
        (0, 100)
        };
    \end{axis}
\end{tikzpicture}
}
\subfloat[$a$]{
\begin{tikzpicture}
    \begin{axis}[xmin=-0.3,xmax=0.3,ymin=0,ymax=20,cycle list name=mycycle,xtick={-0.25,0,0.25},xticklabels={$0.75$,$1$,$1.5$}]
        \addplot+[ybar, bar width=0.0139,fill, opacity=0.6, draw=none] table{Experiments/Experiment3_Discontinuous_flux/hist_discontinuousflux_a_h=128.tikz.txt};
        \addplot+[very thick, domain=-0.6:0.6, samples=100]{1/(0.15*sqrt{(2*pi)})*exp(-0.5*((x+0.1)/0.15)^2)};
        \addplot+[ybar, bar width=.4pt, fill] coordinates{
        (0, 100)
        };
    \end{axis}
\end{tikzpicture}
}
\caption{\emph{Experiment 2:} Histograms corresponding to the Metropolis-Hastings approximation of the posterior using a chain length of $2500$ and $\Dx=2/128$. The ground truth and the prior are shown in red and blue respectively.}
\label{exp3: histograms}
\end{figure}
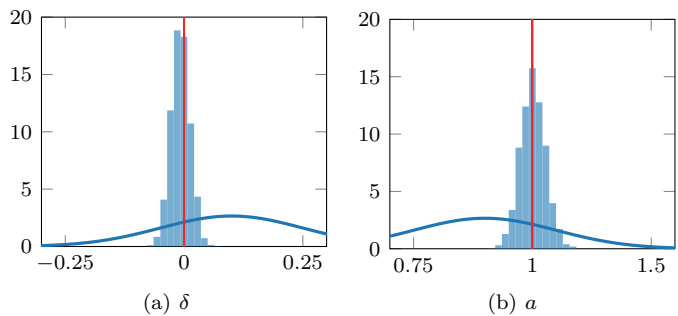
%
%
%
%
%
%
%
%
%
%
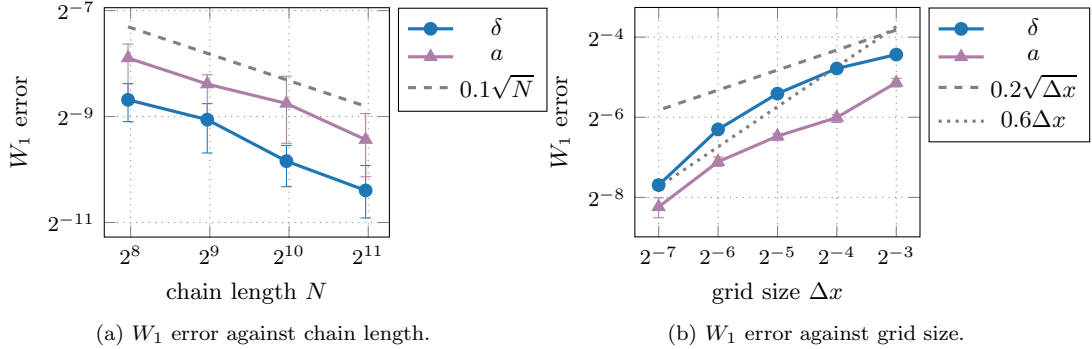
\begin{figure}[t]
\centering
\subfloat[$W_1$ error against chain length.]{
    \begin{tikzpicture}
        \begin{loglogaxis}[
    log basis x={2},
    xlabel={\small chain length $N$},
    log basis y={2},
    ylabel={\small $W_1$ error},
    grid=major,
    grid style={gray, dotted},
    legend pos= outer north east,
    legend style={font=\small},
    cycle list name=mycycle2,
]
            \addplot+[
                very thick,
                error bars/.cd,
                y explicit,
                y dir=both,
            ] table [
                x=steps,
                y=W1,
                y error plus expr=\thisrow{upper_quartile}-\thisrow{W1},
                y error minus expr=\thisrow{W1}-\thisrow{lower_quartile},
            ] {Experiments/Experiment3_Discontinuous_flux/W_1_for_different_chain_lengths_for_delta.tikz.txt};
            \addlegendentry{$\delta$}
            \addplot+[
                very thick,
                error bars/.cd,
                    y explicit,
                    y dir=both,
            ] table [
                x=steps,
                y=W1,
                y error plus expr=\thisrow{upper_quartile}-\thisrow{W1},
                y error minus expr=\thisrow{W1}-\thisrow{lower_quartile},
            ] {Experiments/Experiment3_Discontinuous_flux/W_1_for_different_chain_lengths_for_sigma.tikz.txt};
            \addlegendentry{$a$}
            \addplot[very thick,gray, dashed, domain = 250:2000] {0.1/sqrt(x)};
            \addlegendentry{$0.1\sqrt{N}$}
        \end{loglogaxis}
    \end{tikzpicture}
    \hspace{-1.2em}
}
\subfloat[$W_1$ error against grid size.]{
    \begin{tikzpicture}
        \begin{loglogaxis}[
    log basis x={2},
    xlabel={\small grid size $\Dx$},
    log basis y={2},
    ylabel={\small $W_1$ error},
    grid=major,
    grid style={gray, dotted},
    legend pos= outer north east,
    legend style={font=\small},
    cycle list name=mycycle2,
]
            \addplot+[
                very thick,
                error bars/.cd,
                    y explicit,
                    y dir=both,
            ] table [
                x=cells,
                y=W1,
                y error plus expr=\thisrow{upper_quartile}-\thisrow{W1},
                y error minus expr=\thisrow{W1}-\thisrow{lower_quartile},
            ] {Experiments/Experiment3_Discontinuous_flux/W_1_for_different_grid_spacings_for_delta.tikz.txt};
            \addlegendentry{$\delta$}
            \addplot+[
                very thick,
                error bars/.cd,
                    y explicit,
                    y dir=both,
            ] table [
                x=cells,
                y=W1,
                y error plus expr=\thisrow{upper_quartile}-\thisrow{W1},
                y error minus expr=\thisrow{W1}-\thisrow{lower_quartile},
            ] {Experiments/Experiment3_Discontinuous_flux/W_1_for_different_grid_spacings_for_sigma.tikz.txt};
            \addlegendentry{$a$}
        \addplot[very thick,gray, dashed, domain = 0.125:0.0078125] {0.2*sqrt(x)};
        \addlegendentry{$0.2\sqrt{\Dx}$}
        \addplot[very thick,gray, dotted, domain = 0.125:0.0078125] {0.6*x};
        \addlegendentry{$0.6\Dx$}
        \end{loglogaxis}
    \end{tikzpicture}
}
\caption{\emph{Experiment 2:} $W_1$ error as a function of the chain length with a fixed grid discretization parameter $\Dx=2/128$  (left) and as a function of the grid discretization parameter $\Dx$ for a fixed chain length ($N=2500$) (right)}
\label{fig: convergence rates 2}
\end{figure}

\subsection{An inverse problem for systems of conservation laws.}
While our theory does not cover systems of conservation laws, even in one space dimension, due to a lack of rigorous stability results in the literature, we demonstrate with the following numerical experiment that Bayesian inverse problems for systems of conservation laws, at least in one space dimension, might still be well-approximated with the MCMC type sampling algorithms presented here.

We consider the one-dimensional Euler equations
\begin{gather*}
    w_t + f(w)_x = 0\\
    w=\begin{pmatrix}
        \rho\\
        \rho v\\
        E
    \end{pmatrix},\qquad f(w)=\begin{pmatrix}
        \rho v\\
        \rho v^2+ p\\
        (E+p)v
    \end{pmatrix},
\end{gather*}
where the density $\rho$, velocity $v$ and energy $E$ are unknown and the pressure $p$ and the energy are related by the following equation of state:
\begin{equation*}
    E=\frac{p}{\gamma-1} + \frac{1}{2}\rho v^2,\qquad \text{for }\gamma=1.4.
\end{equation*}
We consider Sod's shock tube problem \cite{sod1978survey} on the domain $(0,1)$ with outflow boundary conditions and initial discontinuity at $x=0.5$. We want to infer the initial datum $\bar{w}^{(\delta_L,\gamma_L,\beta_L,\delta_R,\gamma_R,\beta_R)}$ which we assume takes the left and right states
\begin{equation*}
    \begin{pmatrix}
        \rho_L\\
        v_L\\
        p_L
    \end{pmatrix}
    = \begin{pmatrix}
        1 +\delta_L\\
        \gamma_L\\
        1 +\beta_L
    \end{pmatrix}\qquad\text{and}\qquad
    \begin{pmatrix}
        \rho_R\\
        v_R\\
        p_R
    \end{pmatrix}
    = \begin{pmatrix}
        0.125 + \delta_R\\
        \gamma_R\\
        0.1 + \beta_R
    \end{pmatrix}
\end{equation*}
to the left respectively to the right of the initial discontinuity $x=0.5$. To that end we consider the observation operator
\begin{equation*}
    \left(\mathcal{G}(\delta_L,\gamma_L,\beta_L,\delta_R,\gamma_R,\beta_R)\right)_j = 10 \int_{x_j-0.05}^{x_j+0.05} S_T^\Dx\left(\bar{w}^{(\delta_L,\gamma_L,\beta_L,\delta_R,\gamma_R,\beta_R)})\right)\diff x,\qquad j=1,\ldots,5,
\end{equation*}
at time $T=0.2$ and for the measurement points $(x_j)_{j=1}^5=(0.1,0.25,0.5,0.75,0.9)$. Here $S_T^\Dx$ is a numerical solution operator and in the subsequent experiment we will employ the HLLC method.

We consider observational noise $\eta\sim\mathcal{N}(0,\gamma^2\mathcal{I}_{15})$ with $\gamma=0.05$ and prior $\mu_0\sim\mathcal{N}(u^p,\mathcal{C})$ with mean $u^p=(\delta_L^p,\gamma_L^p,\beta_L^p,\delta_R^p,\gamma_R^p,\beta_R^p)=(-0.1,0.1,-0.1,0.1,0,1,0.1)$ and covariance matrix $\mathcal{C}=\varphi^2\mathcal{I}_6$ with $\varphi=0.15$. The ground truth we want to recover is $u^*=(\delta_L^*,\gamma_L^*,\beta_L^*,\delta_R^*,\gamma_R^*,\beta_R^*)=(0,0,0,0,0,0)$. \Cref{fig: Sod shock tube experiment setup} shows the initial data corresponding to the prior mean and the ground truth as well as corresponding numerical solutions. We chose the constant step size $\beta=0.0005$ and burn-in $b=500$ and sample interval $\tau=10$. The histograms of the approximated posteriors are shown in \Cref{exp4: histograms} and we observe overall good approximation with the only possible exception of the right state of the velocity. Here, we used a chain length of $1500$ and $\Dx=1/128$. The means of the approximated posterior are $u_{\text{mean}}\approx (0.0078, -0.0064, 0.0084, 0.02, 0.046,0.012)$ and the MAP estimators are $u_{\text{MAP}}\approx(0.0092, 0.012, 0.012, 0.014, 0.013, 0.0011)$ both very close to zero.
%
%
%
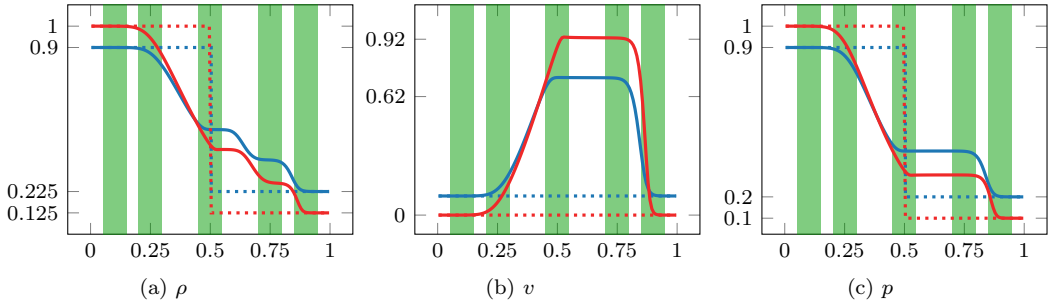
\begin{figure}[t]
\centering
\subfloat[$\rho$]{
\begin{tikzpicture}
  \begin{axis}[ymin=0.025,ymax=1.1,xtick={0, 0.25, 0.5, 0.75,1},xticklabels={$0$, $0.25$, $0.5$, $0.75$, $1$},ytick={0.125,0.225,0.9,1},yticklabels={$0.125$,$0.225$,$0.9$,$1$}]
    \fill[green!60!black, opacity=0.5] (0.05,-10) -- (0.15,-10) -- (0.15,10) -- (0.05,10) -- (0.05,-10);
    \fill[green!60!black, opacity=0.5] (0.2,-10) -- (0.3,-10) -- (0.3,10) -- (0.2,10) -- (0.2,-10);
    \fill[green!60!black, opacity=0.5] (0.45,-10) -- (0.55,-10) -- (0.55,10) -- (0.45,10) -- (0.45,-10);
    \fill[green!60!black, opacity=0.5] (0.7,-10) -- (0.8,-10) -- (0.8,10) -- (0.7,10) -- (0.7,-10);
    \fill[green!60!black, opacity=0.5] (0.85,-10) -- (0.95,-10) -- (0.95,10) -- (0.85,10) -- (0.85,-10);
    \addplot[myblue, dotted, very thick,mark=none, const plot] table {Experiments/Experiment4_Sod_shock_tube/Sod_full_pm_IC_rho.tikz.txt};
    \addplot[myblue, very thick,mark=none] table {Experiments/Experiment4_Sod_shock_tube/Sod_full_pm_end_rho.tikz.txt};
    \addplot[scarletred1, dotted, very thick, mark=none] table {Experiments/Experiment4_Sod_shock_tube/Sod_full_gt_IC_rho.tikz.txt};
    \addplot[scarletred1, very thick,mark=none] table {Experiments/Experiment4_Sod_shock_tube/Sod_full_gt_end_rho.tikz.txt};
  \end{axis}
\end{tikzpicture}
\hspace{-1.3em}
}
\subfloat[$v$]{
\begin{tikzpicture}
  \begin{axis}[ymin=-0.1,ymax=1.1,xtick={0, 0.25, 0.5, 0.75, 1},xticklabels={$0$, $0.25$, $0.5$, $0.75$,$1$},ytick={0,0.62,0.92},yticklabels={$0$,$\phantom{1}0.62$,$0.92$}]
    \fill[green!60!black, opacity=0.5] (0.05,-10) -- (0.15,-10) -- (0.15,10) -- (0.05,10) -- (0.05,-10);
    \fill[green!60!black, opacity=0.5] (0.2,-10) -- (0.3,-10) -- (0.3,10) -- (0.2,10) -- (0.2,-10);
    \fill[green!60!black, opacity=0.5] (0.45,-10) -- (0.55,-10) -- (0.55,10) -- (0.45,10) -- (0.45,-10);
    \fill[green!60!black, opacity=0.5] (0.7,-10) -- (0.8,-10) -- (0.8,10) -- (0.7,10) -- (0.7,-10);
    \fill[green!60!black, opacity=0.5] (0.85,-10) -- (0.95,-10) -- (0.95,10) -- (0.85,10) -- (0.85,-10);
    \addplot[myblue, dotted, very thick,mark=none, const plot] table {Experiments/Experiment4_Sod_shock_tube/Sod_full_pm_IC_v.tikz.txt};
    \addplot[myblue, very thick,mark=none] table {Experiments/Experiment4_Sod_shock_tube/Sod_full_pm_end_v.tikz.txt};
    \addplot[scarletred1, dotted, very thick, mark=none] table {Experiments/Experiment4_Sod_shock_tube/Sod_full_gt_IC_v.tikz.txt};
    \addplot[scarletred1, very thick,mark=none] table {Experiments/Experiment4_Sod_shock_tube/Sod_full_gt_end_v.tikz.txt};
  \end{axis}
\end{tikzpicture}
\hspace{-1.3em}
}
\subfloat[$p$]{
\begin{tikzpicture}
  \begin{axis}[ymin=0.025,ymax=1.1,xtick={0, 0.25, 0.5, 0.75,1},xticklabels={$0$, $0.25$, $0.5$, $0.75$,$1$},ytick={0.1,0.2,0.9,1},yticklabels={$0.1$,$\phantom{15}0.2$,$0.9$,$1$}]
    \fill[green!60!black, opacity=0.5] (0.05,-10) -- (0.15,-10) -- (0.15,10) -- (0.05,10) -- (0.05,-10);
    \fill[green!60!black, opacity=0.5] (0.2,-10) -- (0.3,-10) -- (0.3,10) -- (0.2,10) -- (0.2,-10);
    \fill[green!60!black, opacity=0.5] (0.45,-10) -- (0.55,-10) -- (0.55,10) -- (0.45,10) -- (0.45,-10);
    \fill[green!60!black, opacity=0.5] (0.7,-10) -- (0.8,-10) -- (0.8,10) -- (0.7,10) -- (0.7,-10);
    \fill[green!60!black, opacity=0.5] (0.85,-10) -- (0.95,-10) -- (0.95,10) -- (0.85,10) -- (0.85,-10);
    \addplot[myblue, dotted, very thick,mark=none, const plot] table {Experiments/Experiment4_Sod_shock_tube/Sod_full_pm_IC_p.tikz.txt};
    \addplot[myblue, very thick,mark=none] table {Experiments/Experiment4_Sod_shock_tube/Sod_full_pm_end_p.tikz.txt};
    \addplot[scarletred1, dotted, very thick, mark=none] table {Experiments/Experiment4_Sod_shock_tube/Sod_full_gt_IC_p.tikz.txt};
    \addplot[scarletred1, very thick,mark=none] table {Experiments/Experiment4_Sod_shock_tube/Sod_full_gt_end_p.tikz.txt};
  \end{axis}
\end{tikzpicture}
}
\caption{\emph{Experiment 3:} Initial data (dotted lines) and numerical solutions (solid lines) for $u^p=(\delta_L,\gamma_L,\beta_L,\delta_R,\gamma_R,\beta_R)=(-0.1,0.1,-0.1,0.1,0,1,0.1)$ (blue) corresponding to the prior mean and $u^*=(\delta_L^*,\gamma_L^*,\beta_L^*,\delta_R^*,\gamma_R^*,\beta_R^*)=(0,0,0,0,0,0)$ (red) corresponding to the ground truth. The numerical solutions are calculated using the HLLC scheme and grid discretization parameter $\Dx=1/128$.}
\label{fig: Sod shock tube experiment setup}
\end{figure}
\begin{figure}[t]
\centering
\subfloat[$\rho_L$]{
\begin{tikzpicture}
    \begin{axis}[xmin=-0.3,xmax=0.3,ymin=0,ymax=20,cycle list name=mycycle,xtick={-0.25,0,0.25},xticklabels={$-0.25$,$0$,$0.25$}]
        \addplot+[ybar, bar width=0.0139, fill,opacity=0.6, draw=none] table {Experiments/Experiment4_Sod_shock_tube/Sod_full_post_rho_l_h=128_N=1500.tikz.txt};
        \addplot+[very thick, domain=-0.6:0.6, samples=100]{1/(0.15*sqrt{(2*pi)})*exp(-0.5*((x+0.1)/0.15)^2)};
        \addplot+[ybar, bar width=.4pt, fill] coordinates{
        (0, 100)
        };
    \end{axis}
\end{tikzpicture}
}
\subfloat[$v_L$]{
\begin{tikzpicture}
    \begin{axis}[xmin=-0.3,xmax=0.3,ymin=0,ymax=20,cycle list name=mycycle,xtick={-0.25,0,0.25},xticklabels={$-0.25$,$0$,$0.25$}]
        \addplot+[ybar, bar width=0.0139, fill,opacity=0.6, draw=none] table {Experiments/Experiment4_Sod_shock_tube/Sod_full_post_v_l_h=128_N=1500.tikz.txt};
        \addplot+[very thick, domain=-0.6:0.6, samples=100]{1/(0.15*sqrt{(2*pi)})*exp(-0.5*((x-0.1)/0.15)^2)};
        \addplot+[ybar, bar width=.4pt, fill] coordinates{
        (0, 100)
        };
    \end{axis}
\end{tikzpicture}
}
\subfloat[$p_L$]{
\begin{tikzpicture}
    \begin{axis}[xmin=-0.3,xmax=0.3,ymin=0,ymax=20,cycle list name=mycycle,xtick={-0.25,0,0.25},xticklabels={$-0.25$,$0$,$0.25$}]
        \addplot+[ybar, bar width=0.0139, fill,opacity=0.6, draw=none] table {Experiments/Experiment4_Sod_shock_tube/Sod_full_post_p_l_h=128_N=1500.tikz.txt};
        \addplot+[very thick, domain=-0.6:0.6, samples=100]{1/(0.15*sqrt{(2*pi)})*exp(-0.5*((x+0.1)/0.15)^2)};
        \addplot+[ybar, bar width=.4pt, fill] coordinates{
        (0, 100)
        };
    \end{axis}
\end{tikzpicture}
}
\\
\subfloat[$\rho_R$]{
\begin{tikzpicture}
    \begin{axis}[xmin=-0.3,xmax=0.3,ymin=0,ymax=20,cycle list name=mycycle,xtick={-0.25,0,0.25},xticklabels={$-0.25$,$0$,$0.25$}]
        \addplot+[ybar, bar width=0.0139, fill,opacity=0.6, draw=none] table {Experiments/Experiment4_Sod_shock_tube/Sod_full_post_rho_r_h=128_N=1500.tikz.txt};
        \addplot+[very thick, domain=-0.6:0.6, samples=100]{1/(0.15*sqrt{(2*pi)})*exp(-0.5*((x-0.1)/0.15)^2)};
        \addplot+[ybar, bar width=.4pt, fill] coordinates{
        (0, 100)
        };
    \end{axis}
\end{tikzpicture}
}
\subfloat[$v_R$]{
\begin{tikzpicture}
    \begin{axis}[xmin=-0.3,xmax=0.3,ymin=0,ymax=20,cycle list name=mycycle,xtick={-0.25,0,0.25},xticklabels={$-0.25$,$0$,$0.25$}]
        \addplot+[ybar, bar width=0.0139, fill,opacity=0.6, draw=none] table {Experiments/Experiment4_Sod_shock_tube/Sod_full_post_v_r_h=128_N=1500.tikz.txt};
        \addplot+[very thick, domain=-0.6:0.6, samples=100]{1/(0.15*sqrt{(2*pi)})*exp(-0.5*((x-0.1)/0.15)^2)};
        \addplot+[ybar, bar width=.4pt, fill] coordinates{
        (0, 100)
        };
    \end{axis}
\end{tikzpicture}
}
\subfloat[$p_R$]{
\begin{tikzpicture}
    \begin{axis}[xmin=-0.3,xmax=0.3,ymin=0,ymax=20,cycle list name=mycycle,xtick={-0.25,0,0.25},xticklabels={$-0.25$,$0$,$0.25$}]
        \addplot+[ybar, bar width=0.0139, fill,opacity=0.6, draw=none] table {Experiments/Experiment4_Sod_shock_tube/Sod_full_post_p_r_h=128_N=1500.tikz.txt};
        \addplot+[very thick, domain=-0.6:0.6, samples=100]{1/(0.15*sqrt{(2*pi)})*exp(-0.5*((x-0.1)/0.15)^2)};
        \addplot+[ybar, bar width=.4pt, fill] coordinates{
        (0, 100)
        };
    \end{axis}
\end{tikzpicture}
}
\caption{\emph{Experiment 3:} Histograms corresponding to the Metropolis-Hastings approximation of the posterior using a chain length of $1500$ and $\Dx=1/128$. The ground truth and the prior are shown in red and blue respectively.}
\label{exp4: histograms}
\end{figure}
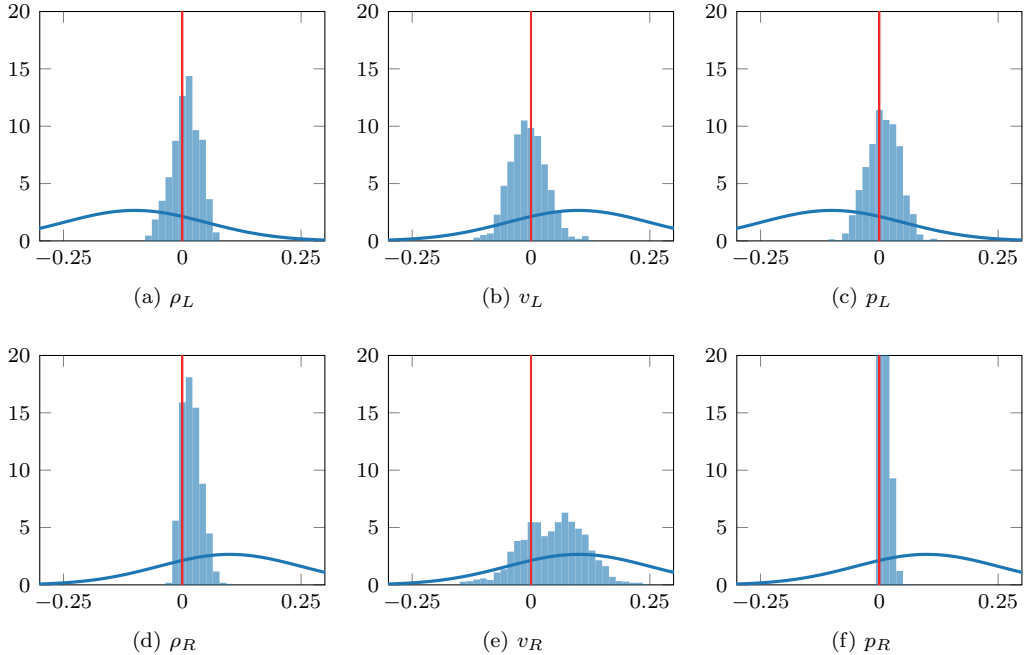
\section{Summary}
We studied the well-posedness of the Bayesian inverse problem for scalar hyperbolic conservation laws in this paper. To this end, we started with the abstract framework for well-posedness that was formalized in \cite{stuart2010inverse}. In contrast to \cite{stuart2010inverse}, we investigated Lipschitz continuity of the measurement to posterior map with respect to the Wasserstein metric. This allows us to more effectively control important statistical moments, such as the means of the posteriors. Moreover, the Lipschitz continuity of \emph{approximate posteriors}, with respect to variations in the approximation parameter, was also derived, allowing us to establish convergence rates with respect to spatio-temporal numerical approximations of the underlying forward map. 

These abstract results were verified for scalar conservation laws, in the context of a Bayesian inverse problem corresponding to inferring the initial datum and flux functions, from noisy measurements of the observables of entropy solutions. Moreover, we also demonstrated the well-posedness of the Bayesian inverse problems for conservation laws with a flux function, that is possibly discontinuous in the space variable. In both cases, explicit stability estimates were obtained for the variation of the posterior in the Wasserstein distance, with respect to measurement perturbations or approximations. 

Finally, we illustrated the theoretical results with numerical experiments, where we verified the convergence rates for the posterior with respect to the spatio-temporal discretization. Our theory and experiments illustrated the fact that the Bayesian inverse problem is both well-posed and can be approximated quite well numerically, even for these nonlinear hyperbolic PDEs with discontinuous solutions. 

Our focus in this paper was on scalar conservation laws as the forward map, in this case, is well-posed and is Lipschitz continuous with respect to the data and to approximations. Extending these results to hyperbolic systems of conservation laws is very challenging. In one space dimension, it is well known that entropy solutions exist and are unique, at least for initial data with small total variation. However, the lack of stability results, particularly with respect to fluxes, inhibits the direct application of our theory in this case. Nevertheless, we presented a numerical experiment to show that the Bayesian inverse problem is computable. However, for systems of conservation laws in several space dimensions, the forward map might not even be globally defined. The well-posedness of the Bayesian inverse problem for such \emph{ill-posed} PDEs is discussed in the recent paper \cite{LMW1}.

\appendix
\section{Appendix}
    \begin{theorem}[{Fernique Theorem \cite[Thm.~2.7]{da2014stochastic}}]\label{thm: Fernique}
    If $\mu=\mathcal{N}(0,\mathcal{C})$ is a Gaussian measure on some Banach space $X$, so that $\mu(X)=1$, then there exists $\alpha>0$ such that
    \begin{equation*}
        \int_X \exp\left(\alpha\norm{x}_X^2\right)\mu(\diff x)<\infty.
    \end{equation*}
\end{theorem}

The Fernique Theorem implies in particular that all moments of $u$ under Gaussian measures $\mu$ are finite as can be seen in the following way:
Since $\norm{u}_X^p = \exp(p \ln\norm{u}_X) \leq \exp(p\norm{u}_X)$ we find
\begin{align*}
    \int_X \norm{u}_X^p \diff\mu(u) &\leq \int_X \exp(p\norm{u}_X)\diff\mu(u)\\
    &= \int_X \left(\exp(p\norm{u}_X)\chi_{\left\{\norm{u}_X\geq\frac{p}{\alpha}\right\}}  + \exp(p\norm{u}_X)\chi_{\left\{\norm{u}_X<\frac{p}{\alpha}\right\}} \right)\diff\mu(u)\\
    &\leq \int_X \exp\left(\alpha\norm{u}_X^2\right)\diff\mu(u) + \exp\left(\frac{p^2}{\alpha}\right)\mu\left( \left\{\norm{u}_X<\frac{p}{\alpha}\right\} \right)\\
    &<\infty.
\end{align*}



\end{document}